\documentclass[letterpaper, 10pt, conference]{ieeeconf}
\usepackage{amsmath,amsfonts}

\IEEEoverridecommandlockouts                              % This command is only needed if 
\makeatother
\usepackage[normalem]{ulem} 
\usepackage{algpseudocode}
\usepackage[]{footmisc}
\usepackage{lipsum}
\usepackage{array}
\usepackage[caption=false,font=footnotesize,labelfont=sf,textfont=sf]{subfig}
\usepackage{textcomp}
\usepackage{stfloats}
\usepackage{url}
\usepackage{verbatim}
\usepackage{graphicx}
\usepackage{cite}
\usepackage{smile}
\usepackage{bbm}
\usepackage{blindtext}
\usepackage[colorlinks,
linkcolor=red,
anchorcolor=blue,
citecolor=blue
]{hyperref}

\newcommand{\xn}[1]{\textcolor{magenta}{#1}}

\def\lb{\bm{\lambda}}
\def\up{\underline{p}}
\def\op{\overline{p}}
\def\uq{\underline{q}}
\def\oq{\overline{q}}

\begin{document}
\title{DISH: A Distributed Hybrid Primal-Dual Optimization Framework to Utilize System Heterogeneity}

\author{Xiaochun Niu and Ermin Wei
        % <-this % stops a space
\thanks{The authors are with the Department of Industrial Engineering and Management Sciences, Northwestern University, Evanston, IL 60201 USA.}% <-this % stops a space
\thanks{This work was supported in part by the National Science Foundation (NSF) under Grant ECCS-2030251 and CMMI-2024774.}}



\maketitle

\begin{abstract}
We consider solving distributed consensus optimization problems over multi-agent networks. Current distributed methods fail to capture the heterogeneity among agents' local computation capacities. We propose DISH as
a \textit{\underline{dis}tributed \underline{h}ybrid primal-dual algorithmic framework} to handle and utilize system heterogeneity. Specifically, DISH allows those agents with higher computational capabilities or cheaper computational costs to implement Newton-type updates locally, while other agents can adopt the much simpler gradient-type updates. We show that DISH is a general framework and includes EXTRA, DIGing, and ESOM-0 as special cases. Moreover, when all agents take both primal and dual Newton-type updates, DISH approximates Newton's method by estimating both primal and dual Hessians. Theoretically, we show that DISH achieves a linear (Q-linear) convergence rate to the exact optimal solution for strongly convex functions, regardless of agents' choices of gradient-type and Newton-type updates. Finally, we perform numerical studies to demonstrate the efficacy of DISH in practice. To the best of our knowledge, DISH is the first hybrid method allowing heterogeneous local updates for distributed consensus optimization under general network topology with provable convergence and rate guarantees.
\end{abstract}

%%%%%%%%%%%%%%%%%%%%%%%%%%%%%%%%%%%%%%%%%%%%%%%%%%%%%%%%%%%%%%%

%!TEX root = ieee_tsp/main.tex

\section{Introduction}
Distributed optimization problems over a connected network with multiple agents have gained significant attention recently. This is motivated by a wide range of applications such as power grids \cite{giannakis2013monitoring,dorfler2015breaking}, sensor networks \cite{ling2010decentralized,schizas2007consensus}, communication networks \cite{lamnabhi2017systems,yang2019survey}, and machine learning \cite{sayed2014adaptation,warnat2021swarm}. In such problems, each agent only has access to its local data and only communicates with its neighbors in the network due to privacy issues or communication budgets \cite{konevcny2016federated}. 
All agents in the system aim to optimize an objective function collaboratively by employing a distributed procedure. Formally, we denote by $\cG=\{\cN,\cE\}$ a connected undirected network with the node set $\cN=\{1,\cdots,n\}$ and the edge set $\cE\subseteq \{\{i,j\}\given i,j\in\cN, i\neq j\}$. We study the distributed optimization problem over $\cG$,
\#\label{prob:original}
    \min_{\omega\in \RR^d} \quad\sum\limits_{i=1}^n f_i(\omega),
\#
where $\omega\in\RR^d$ is the decision variable and $f_i:\RR^d\to \RR$ is the local objective function corresponding to the $i^{th}$ agent. For instance, if we consider an empirical risk minimization problem in a supervised learning setting, the goal of the system is to learn a shared model over all the data in the network without exchanging local data, where local $f_i$ denotes expected loss over the local data at the $i^{th}$ agent.

In order to develop a distributed method for solving Problem \ref{prob:original}, we decouple the computation of individual agents by introducing the local copy of the decision variable at the $i^{th}$ agent as $x_i \in \RR^d$. We formulate Problem \ref{prob:original} over the network $\cG$ as a \textit{consensus optimization} problem \cite{bertsekas1989parallel,nedic2009distributed},
\#\label{prob:constrained_intro}
    \min_{x_1,\cdots,x_n}  \  \sum\limits_{i=1}^n f_i(x_i) \ \
    \text{s.t. } x_i = x_j, \text{ for } \{i,j\}\in\cE.
\#
The consensus constraint $x_i = x_j$ for $\{i,j\}\in\cE$ enforces the equivalence of  Problems \ref{prob:original} and \ref{prob:constrained_intro} for a connected network $\cG$.

While there is growing literature on developing distributed optimization algorithms to solve Problem \ref{prob:constrained_intro}, most existing methods require all agents to take the same type of updates. Such methods include gradient-type methods \cite{nedic2009distributed,shi2015extra,nedic2017achieving,qu2017harnessing} and Newton-type methods \cite{shamir2014communication,zhang2015disco,wang2018giant,crane2019dingo}. With these methods, if any agent in a system cannot handle high-order computation, a fast-converging method utilizing higher-order information will not be applicable to the whole system. As a result, the system could not fully utilize the distributed computation capability when faced with heterogeneity. This is in stark contrast to the fact that many practical distributed systems have heterogeneous agents. There can be drastically varying computation and communication capabilities among the agents due to different hardware, network connectivity, and battery power. \cite{chen2015mxnet}. 
\begin{figure}[!t]
\centering
\includegraphics[width=0.7\linewidth]{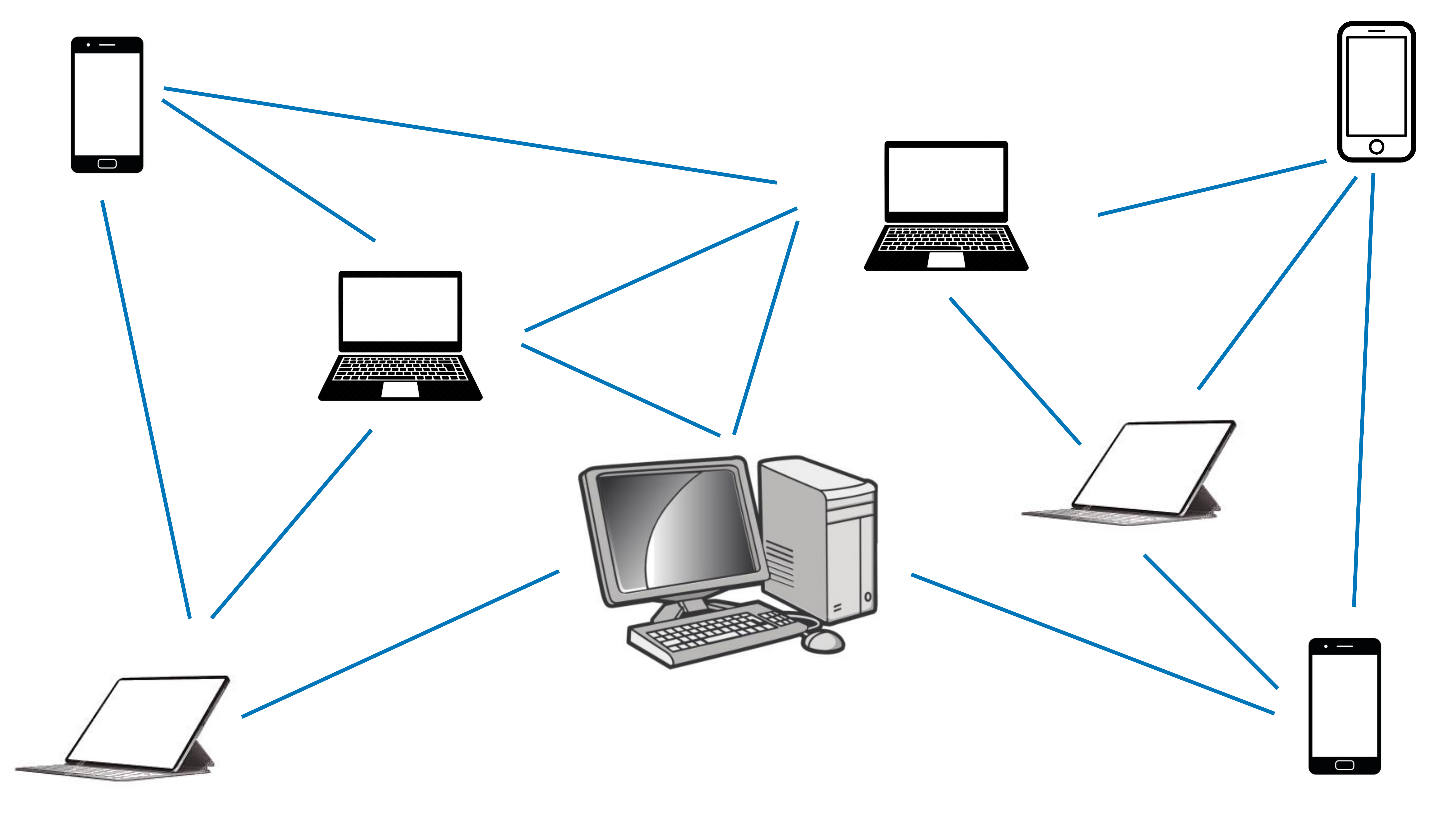}
     \caption{A Heterogeneous System.}
     \label{fig:Hetero}
\end{figure}
Figure \ref{fig:Hetero} shows an example of a heterogeneous system. Moreover, due to the recent global chip shortage, processors with advanced computation capability have very limited availability. 
Consequently, many distributed computation systems have only a few agents with advanced hardware co-existing with many older processors. Therefore, it is imperative to provide a flexible and efficient hybrid method to utilize heterogeneous agents. To the best of our knowledge, this paper takes the first step in this direction. 

In order to handle and utilize the system heterogeneity, we propose a \textit{\underline{dis}tributed \underline{h}ybrid primal-dual algorithmic framework} named DISH. DISH allows agents to choose gradient-type updates or Newton-type updates based on their computation capabilities. Specifically, there can be both gradient-type and Newton-type agents in the same communication round and each agent can switch to either type of updates based on its current situation. We show that DISH include primal-dual gradient-type methods such as EXTRA \cite{shi2015extra}, DIGing \cite{nedic2017achieving}, and \cite{qu2017harnessing} and primal-Newton-dual-gradient methods like ESOM \cite{mokhtari2016decentralized} as special cases. %\ew{Do we want the next sentence? Maybe say this after simulation and say that offers improvement over first order methods?}
 %approximates the Newton's method by estimating both primal and dual Hessians. 
Theoretically, we show that DISH achieves a linear (Q-linear) convergence rate to the exact optimal solution for strongly convex functions, regardless of agents' choices of gradient-type and Newton-type updates. Finally, we conduct numerical experiments on decentralized least squares problems and logistic regression problems and demonstrate the efficacy of the DISH algorithmic framework. We observe that when all agents always take primal-dual Newton-type updates, DISH offers faster convergence speed over gradient methods.

\vskip4pt
\noindent\textbf{Related Works. }Our work is related to the proliferating literature on distributed optimization methods to solve Problem \ref{prob:constrained_intro}.
There are first-order primal iterative methods, like distributed (sub)-gradient descent (DGD) \cite{nedic2009distributed}, which takes a linear combination of a local gradient descent step and a weighted average among local neighbors. DGD finds a near-optimal solution with constant stepsize. %\ew{maybe not the rest of the sentence, since it's unclear} \sout{since it can be viewed as a descent method with respect to a quadratic penalty problem.} 
Based on DGD, other related methods including \cite{shi2015extra,nedic2017achieving,qu2017harnessing,jakovetic2018unification} use gradient tracking technique, which can find the exact solution with constant stepsize and be viewed as primal-dual gradient methods with respect to augmented Lagrangian formulation. 
Second-order primal methods, including Network Newton \cite{7094740} and Distributed Newton method \cite{tutunov2019distributed}, rely on an inner loop to iteratively approximate a Newton step.  \cite{9069966} derives a DGD based method with the inclusion of first and second-order updates in the continuous-time setting. Their method cannot be directly applied in discrete-time and lacks convergence rate analysis. Another popular approach is to use dual decomposition-based methods such as ADMM \cite{boyd2011distributed,wang2019global}, CoCoA\cite{smith2018cocoa}, ESOM \cite{mokhtari2016decentralized}, and PD-QN \cite{eisen2019primal}. Among these, PD-QN is a primal-dual quasi-Newton method with a linear convergence guarantee. ESOM is most related to our approach, which proposes to perform second-order updates
in the primal space and first-order updates in the dual space and has a provable linear convergence rate. However, none of these methods allow different types of updates for heterogeneous agents.
Our earlier work \cite{niu2021fedhybrid} develops a linearly converging distributed primal-dual hybrid method that allows different types of updates, but relies on the structure of a server-client (federated) network. To the best of our knowledge, DISH is the first hybrid method allowing heterogeneous local updates for distributed consensus optimization under general network topology with provable convergence and rate guarantees.

\vskip4pt
\noindent\textbf{Contributions. } Our main contributions are fourfold:
\begin{itemize}
    \item We propose DISH as a distributed hybrid primal-dual algorithmic framework, which allows agents to employ both gradient-type and Newton-type information to harvest system heterogeneity.
        \item For the agents capable of second order computation in DISH, we develop a Newton-type method that approximates Newton's step in both the primal and the dual spaces with a distributed implementation.
        \item We show a linear convergence analysis of DISH to find the optimal solution regardless of agents' choice of gradient-type or Newton-type updates.
        \item We conduct numerical experiments and demonstrate the efficacy of DISH in practice.
\end{itemize}
%Due to space considerations, proofs have been omitted. Interested readers are referred to Section Appendix for details at \href{https://github.com/xniu1/DISH/blob/main/Appendix.pdf}{https://github.com/xniu1/DISH/blob/main/Appendix.pdf}.

\vskip4pt
\noindent\textbf{Notations.} We denoted by $\otimes$ the Kronecker product. For any $m\in \ZZ^+$, we denote by $I_m\in\RR^{m\times m}$ the identity matrix and  $ \mathds{1}_m= \rbr{1,\cdots,1}^\intercal\in\RR^m$ the vector of all ones. 
For any symmetric matrix $S$, we denote by $\rho(S)$ its spectral radius. For any positive semidefinite matrix $M\in\RR^{p\times p}$, we denote by $\sigma_{\min}^+(M)$ its smallest positive eigenvalue and $\|y\|_{M}^2=y^\intercal My$ for any $y\in\RR^{p}$. For any positive definite matrix $A$, we denote by $\theta_{\min}(A)$ its smallest eigenvalue.

%!TEX root = ieee_tsp/main.tex

\section{Preliminaries}
In this section, we reformulate Problem \ref{prob:constrained_intro} in a compact form and introduce its dual problem {based on the augmented Lagrangian} \cite{bertsekas2014constrained}, which prepares our derivation of DISH.

\vskip4pt
\noindent\textbf{Equivalent 
Reformulation.}
For compactness, we reformulate Problem \ref{prob:constrained_intro} in the following equivalent form,
\#\label{prob:constrained_pre}
    \min_{\xb\in\RR^{nd}}  \quad f(\xb) = \sum_{i=1}^n f_i(x_i) \quad
    \text{s.t. } (Z\otimes I_d) \xb = \xb, 
\#
where $\xb = \rbr{x_1^\intercal,\cdots,x_{n}^\intercal}^\intercal$ is the concatenation of local variables,  $f:\RR^{nd}\to\RR$ is the aggregate function, and $Z\in \RR^{d \times d}$ with elements $z_{ij}$ is a consensus matrix corresponding to $\cG$. We emphasize that $Z$ satisfies the following assumption.
\begin{assumption}\label{ass:consensus} The consensus matrix $Z$ satisfies that:
\begin{itemize}
    \item[(a)] Off-diagonal elements: $z_{ij}\neq 0$ if and only if $\{i,j\} \not\in \cE$;
    \item[(b)] Diagonal elements: $z_{ii}>0$ for all $i\in\cN$;
    \item[(c)] $z_{ij}=z_{ji}$ for all $i\neq j$ and $i,j\in\cN$;
    \item[(d)] $Z\mathds{1}_n = \mathds{1}_n$.
\end{itemize}
\end{assumption}
Assumption \ref{ass:consensus} is standard for consensus matrices, where (a) states the right sparsity pattern of $Z$, (b) ensures the aperiodicity of $\cG$, and (c) and (d) impose that $Z$ is symmetric and doubly stochastic. We denote by $\gamma$ the second largest eigenvalue of $Z$. With the irreducibility of $Z$ guaranteed by the connectness of $\cG$, by Perron-Frobenius theorem, we have  $\rho(Z)=1$, $\gamma<1$, and $\text{null}(I-Z) = \text{span}\{\mathds{1}_{n}\}$. A matrix $Z$ under Assumption \ref{ass:consensus} is known as the consensus matrix due to its property that $(Z\otimes I_d) \xb = \xb$ if and only if $x_i = x_j$ for all $i,j\in\cN$ \cite{nedic2009distributed}. If we denote $W=(I_{n} - Z)\otimes I_d $, we have $\sigma_{\min}^+(W)=1-\gamma$ and $\text{null}(W) = \text{span}\{\mathds{1}_{n}\otimes y:y\in\RR^{d}\}$. We can rewrite Problem \ref{prob:constrained_pre} using the matrix $W$ as follows,
\#\label{prob:constrained}
    \min_{\xb\in\RR^{nd}}  \quad f(\xb) = \sum_{i=1}^n f_i(x_i) \quad
    \text{s.t. } W \xb = 0. 
\#
We will impose the next assumption % on local functions 
throughout the paper.
\begin{assumption}
\label{ass:hessian}
The local function $f_i$ is twice differentiable, $s_i$-strongly convex, and $\ell_i$-Lipschitz smooth with positive constants $0 < {s_i} \le \ell_i <\infty$ for any agent $i\in\cN$.
\end{assumption}

Assumption \ref{ass:hessian} postulates that the local Hessian is bounded by ${s_i} I_d \preceq  \nabla^2f_i(\cdot)\preceq {\ell_i} I_d$ for any $i \in\cN$.
For convenience, we denote by $s = \min_{i\in\cN}\{s_i\}$ and $\ell = \max_{i\in\cN}\{\ell_i\}$.

\vskip4pt
\noindent\textbf{Augmented Lagrangian and Dual Problem.} In order to develop primal-dual methods for solving Problem \ref{prob:constrained} with the consensus constraint, we introduce the dual function based on the augmented Lagrangian. 
We denote by $\lb = \rbr{\lambda_1^\intercal,\cdots, \lambda_n^\intercal}^\intercal$ the dual variable with $\lambda_i\in\RR^d$ associated with the constraint $z_{ii}x_i - \sum_{j\in\cN}z_{ij}x_j=0$ at agent $i$. We define the augmented Lagrangian $L(\xb, \lb)$ of Problem \ref{prob:constrained} as
\#\label{eq:lagrangian_func}
L(\xb,\lb) = f(\xb) + \lb^{\intercal}W \xb + \dfrac{\mu}{2}\xb^{\intercal}V\xb,
\#
where $\mu\ge0$ and $V\in\RR^{nd\times nd}$ is positive semi-definite with $\text{null}(V) = \text{span}\{\mathds{1}_{n}\otimes y:y\in\RR^{d}\}$. The augmentation term $\mu \xb^{\intercal}V\xb / 2$ serves as a penalty for the violation of the consensus constraint. Examples of choices for $V$ include $W$ and $W^2$. For convenience, throughout the paper, we will use $V=W$. We remark that $L(\xb,\lb)$ is the (unaugmented) Lagrangian function when $\mu=0$. The augmented Lagrangian in \eqref{eq:lagrangian_func} can also be viewed as the (unaugmented) Lagrangian associated with the penalized problem 
\#\label{prob:penalized}
\min_{\xb\in\RR^{nd}}  f(\xb) + \frac{\mu}{2}{\xb}^{\intercal}V{\xb}\quad \text{ s.t. }W\xb=0.
\# Problem \ref{prob:penalized} is equivalent to Problem \ref{prob:constrained} since $\mu \xb^{\intercal}V\xb / 2$ is zero for any feasible $\xb$. By the convexity condition in Assumption \ref{ass:hessian} and the Slater’s condition,  strong duality holds for Problem \ref{prob:penalized} \cite{boyd2004convex}. Thus, Problem \ref{prob:penalized}, as well as Problem \ref{prob:constrained}, are  equivalent to the following dual problem,
\#\label{prob:dual}
\max_{\lb\in\RR^{nd}} g(\lb), \text{ where } g(\lb) = \min_{\xb\in\RR^{nd}} L(\xb, \lb),
\#
where we refer to $g:\RR^{nd}\to\RR$ as the dual function. For any $\lb\in\RR^{nd}$, as we will show in Lemma \ref{lem:L_property}, the function $L(\cdot,\lb)$ is strongly convex with a unique minimizer defined as
\#\label{eq:x_star}
\xb^*(\lb)= \argmin_{\xb\in\RR^{nd}} L(\xb,\lb).
\#
%, that is, there is no duality gap between the primal Problem \ref{prob:constrained} and its dual Problem \ref{prob:dual}. 
By the definition of $g$ in \eqref{prob:dual}, we have $L(\xb^*(\lb), \lb) =g(\lb)$.
We show the explicit forms of the gradient and the Hessian of the dual function $g$ in the following lemma \cite{tapia1977diagonalized}.
\begin{lemma} \label{lem:dual_hessian}
Under Assumption \ref{ass:hessian}, with $ \xb^*(\lb)$ defined in \eqref{eq:x_star}, the gradient and the Hessian of the dual function $g(\lb)$ defined in Problem \ref{prob:dual} are given by
\$
&\nabla g(\lb) = W\xb^*(\lb), \\
&\nabla^2 g(\lb) =  -W(\nabla_{\xb\xb}^2 L(\xb^*(\lb),\lb))^{-1}W.
\$
\end{lemma}
For the rest of the paper, we focus on developing distributed methods for solving Problem \ref{prob:dual}.

\section{Algorithm}
This section proposes DISH as a distributed hybrid primal-dual algorithmic framework for solving Problem \ref{prob:dual}, which allows choices of gradient-type and Newton-type updates for each agent at each iteration based on their current battery/computation capabilities and provides flexibility to handle and utilize heterogeneity in the network.

\subsection{DISH to Handle and Utilize System Heterogeneity}
%Since the agents in the network may be heterogeneous, we propose DISH as a hybrid primal-dual algorithmic framework that can handle and utilize system heterogeneity. 
Specifically, %all agents in the network can choose to perform gradient-type or Newton-type updates based on their computation capacities, and they can switch their choices during the iterations of the algorithm.
%Based on such options of different update types, 
we propose the following hybrid updates. At each iteration $k$,
\#\label{eq:up_pre}
& \xb^{k+1} = \xb^k - AP^k\nabla_{\xb}L(\xb^k,\lb^k), \notag \\
& \lb^{k+1}=\lb^k + BQ^k\nabla_{\lb} L (\xb^k,\lb^k),
\#
where stepsize matrices $A = \diag\{a_1, \cdots, a_n\}\otimes I_d$ and $B = \diag\{b_1,\cdots,b_n\}\otimes I_d$ consist of personalized stepsizes $a_i$ and $b_i>0$ for $i\in\cN$ and block diagonal update matrices $P^k = \diag\{P_1^k,\cdots,P_n^k\}$ and $Q^k = \diag\{Q_1^k,\cdots,Q_n^k\}$ are composed of positive definite local update matrices $P_i^k$ and $Q_i^k\in\RR^{d\times d}$ for $i\in\cN$. Here are some examples  of possible local update matrices. For the primal updates, we can take
\#\label{eq:primal_p_choice}
&\text{Gradient-type: }P_i^k=I_d; \notag\\
&\text{Newton-type: }P_i^k = (\nabla^2 f_i(x_i^k) + \mu I_d)^{-1}.
\#
As for the dual updates, we can use
\#\label{eq:dual_q_choice}
&\text{Gradient-type: }Q_i^k=I_d; \notag\\
&\text{Newton-type: }Q_i^k = \nabla^2 f_i(x_i^k) + \mu I_d.
\#
%$P_i^k=I_d$ and $P_i^k = (\nabla^2 f_i(x_i^k) + \mu I_d)^{-1}$ for gradient-type and Newton-type updates, respectively. As for dual updates, we can use $Q_i^k=I_d$ and $Q_i^k = \nabla^2 f_i(x_i^k) + \mu I_d$ for gradient-type and Newton-type updates, respectively. 
As we go through the following sections, we will explain such choices of local update matrices. We remark that as we will show in Theorem \ref{thm:strong}, the analysis of DISH only requires the update matrices $P_i^k$ and $Q_i^k$ to be positive definite. Thus, the agents can take other local updates such as quasi-Newton methods like BFGS \cite{nocedal1999numerical} and the scaled gradient method \cite{bertsekas1989parallel}. These can be future directions. 
Nevertheless, this paper mainly focuses on gradient-type and Newton-type updates. 

By substituting the partial derivatives in \eqref{eq:up_pre}, we can write DISH in a compact form as follows, at iteration $k$,
\#\label{eq:up}
& \xb^{k+1} = \xb^k - AP^k(\nabla f(\xb^k) + W \lb^k + \mu W\xb^k), \notag \\
& \lb^{k+1}=\lb^k + BQ^kW\xb^k,
\#
Based on \eqref{eq:up}, we provide the distributed implementation of DISH in Algorithm \ref{alg:dish}. DISH in Algorithm \ref{alg:dish} consists of a primal (Line \ref{line:primal}) and a dual step (Line \ref{line:dual}) for each agent, where both steps can be either gradient-type or Newton-type based on the agent's choice in each iteration. We note that this is a very flexible framework. An agent may use different types of updates across iterations, and between primal and dual spaces within the same iteration. The primal and dual updates can be computed simultaneously as they both depend on values from the previous iteration.
%We remark that the agents  have the freedom to switch their local updates between gradient-type and Newton-type ones with appropriate stepsize selections. 
In the sequel, we provide interpretation on DISH with some specific choices of updates.

\begin{algorithm}[!t] 
    \caption{DISH: Distributed Hybrid Primal-dual Algorithmic Framework for Consensus Optimization}
    \label{alg:dish}
    \begin{algorithmic}[1]
    \State{\textbf{Input:} Initialization $x_i^0,\lambda_i^0\in\RR^d$, stepsizes $a_i,b_i\in \RR^+$ for all $i\in \cN$, and the penalty parameter $\mu$.}  
    % \State{{Each agent chooses its update type. $J_1, J_2$ are constructed.}}
    % \State{Choose constant $\mu>0$ and stepsizes $a_i$, $b_i$ for all $i\in\cN$} \algorithmiccomment{See Theorem \ref{thm:strong}, \ref{thm:general} for parameter choices}
    \For{$k =  0, \ldots, K-1$} 
        \For{each agent $i\in \cN$ in parallel}%\algorithmiccomment{Newton-type} \label{line:n_start}
            \State{Send $x_i^{k}$ and $\lambda_i^{k}$ to its neighbors $j$ for $\{i,j\}\in\cE$;}
           \State{\vspace{-0.5cm} \begin{align*}\qquad 
           \quad &x_i^{k+1} = x_i^k - a_iP_i^k\bigg\{\nabla f_i(x_i^k) + (1-z_{ii})\lambda_i^k \\
&           -\sum_{\{j,i\}\in\cE}z_{ji}\lambda_{j}^k + \mu\big[(1-z_{ii})x_i^k  -\sum_{\{i,j\}\in\cE}z_{ij}x_{j}^k\big]\bigg\};\end{align*} \label{line:primal}}
           \State{$\lambda_i^{k+1} = \lambda_i^k + b_iQ_i^k\big[(1-z_{ii})x_i^k - \sum_{\{i,j\}\in\cE}z_{ij}x_{j}^k\big]$;}\label{line:dual}
        \EndFor
    \EndFor
\end{algorithmic}
\end{algorithm}

\subsection{Relation of DISH to Existing Methods}\label{sec:other_methods}
Now we illustrate how the introduced DISH algorithm is related to some other distributed optimization methods.
\subsubsection{Primal-Dual Gradient-type Method (EXTRA, DIGing, and \cite{qu2017harnessing})}
When all agents in the  network perform primal and  dual  gradient-type  updates, that is, $P^k=Q^k=I_{nd}$, the compact form \eqref{eq:up} of Algorithm \ref{alg:dish} at iteration $k$ is as follows,
\#\label{eq:up_1st}
&\xb^{k+1} = \xb^k - A(\nabla f(\xb^k) + W \lb^k + \mu W\xb^k),  \notag \\
&\lb^{k+1} =\lb^k + BW\xb^k.
\#
This recovers the Arrow-Hurwicz-Uzawa method \cite{arrow1958h}. For convenience, we refer to updates in \eqref{eq:up_1st} as DISH-G. We remark that some exact distributed first-order methods with gradient tracking techniques, such as EXTRA \cite{shi2015extra}, DIGing \cite{nedic2017achieving}, and \cite{qu2017harnessing}, are also equivalent to primal-dual gradient-type methods similar to DISH-G \cite{jakovetic2018unification}. The only difference between these methods and DISH-G occurs in different choices of consensus constraints or penalty terms used in the augmented Lagrangian, or whether the dual step adopts the previous primal variable $x^k$ or the updated $x^{k+1}$ (also referred to as Jacobi or Gauss-Seidel updates).

While such gradient-type primal-dual methods lead to simple distributed implementation, they suffer from slow convergence due to their first-order nature. This motivates us to involve Newton-type updates in DISH as a speedup. %Next, we will propose a Newton-type approximation of MM under the server-agent topology.

\subsubsection{Primal-Newton-Dual-Gradient Method (ESOM  \cite{mokhtari2016decentralized})} 
ESOM is a second-order method that each iteration approximates a Newton's step by an inner loop in the primal space and a gradient ascent step in the dual space. ESOM-0 is a variant of ESOM without the primal inner loop and can be viewed as a special case of DISH with a different choice of positive definite update matrix. In particular, we define $P^k_{\text{ESOM}}= \nabla^2 f(\xb^k) + \mu (I_{n} - \diag(Z))\otimes I_d$, which is a diagonal approximation of the primal Hessian $\nabla_{\xb\xb}^2L_{\mu}(\xb^k,\lb^k)= \nabla^2{f}(\xb^k) + \mu W$ when $\mu>0$ and exact when $\mu=0$.  When all agents in DISH perform primal Newton-type and dual gradient-type updates with $P^k = P^k_{\text{ESOM}}$ and $Q^k = I_{nd}$, respectively, the updates of DISH in \eqref{eq:up} at iteration $k$ is
\$
&\xb^{k+1} = \xb^k - AP^k_{\text{ESOM}}(\nabla f(\xb^k) + W \lb^k + \mu W\xb^k),  \notag \\
&\lb^{k+1} =\lb^k + BW\xb^k,
\$
which coincide with ESOM-0. Other variants of ESOM iteratively approximates the off-diagonal parts of the primal Hessian. ESOM enjoys the speedup brought by the primal Hessian's information. Our numerical study shows that DISH with Newton-type updates in both primal and dual spaces can further benefit from the dual Hessian's information.

\subsection{DISH-N as an Approximated Newton's Method} 
Now we take a close inspection of DISH when all agents in the network always take both primal and dual Newton-type updates. In this case, DISH in \eqref{eq:up} shows as follows,
\#\label{eq:up_2nd}
& \xb^{k+1} = \xb^k - {A}(H^k)^{-1}(\nabla f(\xb^k) + W \lb^k + \mu W\xb^k), \notag \\
& \lb^{k+1}=\lb^k + BH^kW{\xb}^k,
\#
where $H^k = \nabla^2 f(\xb^k) + \mu I_{nd}= \diag\{\nabla^2 f_1(x_1^k) + \mu I_{d},\cdots,\nabla^2 f_n(x_n^k) + \mu I_{d}\}$ approximates the primal Hessian $\nabla_{\xb\xb}^2L_{\mu}(\xb^k,\lb^k)$. We will refer to updates in \eqref{eq:up_2nd} as DISH-N. In the sequel, we present that DISH-N can be viewed as an approximation of a Newton's method that takes Newton's steps in both the primal and the dual spaces.
\vskip4pt
\noindent\textbf{Primal Update.} 
The primal Newton's step for solving the inner problem $\min_{\xb}L(\xb, \lb^k)$ in \eqref{prob:dual} at iteration $k$ is as follows,
\#\label{eq:up_2nd_app_p}
\xb^{k+1} = \xb^k -\big(\nabla_{\xb\xb}^2L(\xb^k,\lb^k)\big)^{-1}\nabla_{\xb}{L}(\xb^k,\lb^k).
\#
We note that when $\mu=0$, the primal update in \eqref{eq:up_2nd} coincides with the exact Newton's step in \eqref{eq:up_2nd_app_p}. As when $\mu>0$, the primal Hessian $\nabla_{\xb\xb}^2L(\xb^k,\lb^k) = \nabla^2 f (\xb^k) + \mu W$ is nonseparable due to the penalty term $\mu W$, which makes it difficult to compute the exact Hessian inverse  in a distributed way. Here we approximate $W= (I_n-Z)\otimes I_d$ by the identity $I_{nd}$. We remark that $(I_n-\diag(Z))\otimes I_d$ used in $P^k_{\text{ESOM}}$ is another approximation of $W$. Adopting either of them is guaranteed to provide linear convergence rate by Theorem \ref{thm:strong}. %, which leaves aside the communication in the penalty term. 
By substituting $H^k = \nabla^2 f(\xb^k) + \mu I_{nd}$ in \eqref{eq:up_2nd_app_p}, we obtain the Newton-type primal update in \eqref{eq:up_2nd}.

\vskip4pt
\noindent\textbf{Dual Update.} 
Now we consider the dual Newton's update for $\max_{\lb} g(\lb)$ in \eqref{prob:dual} at iteration $k$. Since we cannot get the exact primal minimizer $\xb^*(\lb^k)$ used in $\nabla g(\lb^k)$ and $\nabla^2 g(\lb^k)$ by Lemma \ref{lem:dual_hessian}, we replace $\xb^*(\lb^k)$ by the current primal iterate $\xb^k$ and define $\hat \nabla g(\lb^k)$ and $\hat \nabla^2 g(\lb^k)$ as estimators of $\nabla g(\lb^k)$ and $\nabla^2 g(\lb^k)$, respectively, as follows, 
\$
\hat \nabla g(\lb^k) = W\xb^k, \ \hat\nabla^2 g(\lb^k) = -W(\nabla_{\xb\xb}^2L(\xb^k,\lb^k))^{-1}W.
\$
We remark that $\hat\nabla^2 g(\lb^k)$ is not full-rank due to the matrix $W$.
We denote by $\lb^{k+1} = \lb^{k} + \Delta\lb^k$ the dual update at iteration $k$, where $\Delta \lb^k$ is an approximated dual Newton's step satisfying
\#\label{eq:up_2nd_d}
\hat \nabla^2 g(\lb^k) \Delta \lb^k = \hat \nabla g(\lb^k).
\#
We define a Hessian weighted average of local primal variables $y^k$ as follows,
\#\label{eq:y_k}
y^k = \big(\sum_{i\in\cN}\nabla^2f_i(x_i^k) \big)^{-1}\sum_{i\in\cN}\nabla^2f_i(x_i^k) x_i^k.
\#
Now we introduce a lemma to characterize $\Delta \lb^k$ using $y^k$.
\begin{lemma}\label{lem:dual_hessian_app}
Under Assumption \ref{ass:hessian}, with $y^k$ defined in \eqref{eq:y_k}, the dual Newton's step $\Delta \lb^k$ in \eqref{eq:up_2nd_d} satisfies
\$
W \Delta\lb^k & = \nabla_{\xb\xb}^2L(\xb^k,\lb^k)\rbr{\mathds{1}_n\otimes y^k - \xb^k}.
\$
\end{lemma}

We note that calculating $y^k$ in \eqref{eq:y_k} directly is impractical in a distributed manner since communicating $d\times d$ local Hessians can be prohibitively expensive. Thus, at the $i^{th}$ agent, we estimate $y^k$ by the weighted average of $x_j^k$ from its neighbors in the network, that is, $\sum_{j\in\cN} z_{ij}x_j^k = [(Z\otimes I_d)x^{k}]_{i}$. Such estimators are more accurate when either the local Hessians are similar to each other or the local decision variables $x_i^k$ are close to consensus. This includes the scenarios when the original problem is generated by an empirical risk minimization problem with i.i.d. samples at each node, or when the underlying graph has good algebraic connectivity or when the method is close to its limit point (a consensed point). %Consensus happens quickly if the underlying graph has small spectral gap and $Z=\mathds{1}_n^\intercal\mathds{1}_n/n$. As primal variables come to the consensus, the estimators are more accurate when the network is closer to the complete graph and agents access the i.i.d. local data, where local functions are more similar.

If we substitute $(Z\otimes I_d)\xb^k$ and $H^k=\nabla^2 f(\xb^k) + \mu I_{nd}$ as estimators of $\mathds{1}_n \otimes y^k$ and $\nabla^2_{\xb\xb} L(\xb^k, \lb^k)$ in Lemma \ref{lem:dual_hessian_app}, respectively, we obtain an estimator $\Delta\hat\lb^k$ of $\Delta \lb^k$ satisfying
\$
W\Delta\hat\lb^k = {H}^k((Z\otimes I_d)\xb^k - \xb^k) = -H^k W\xb^k.
\$ 
We highlight that only $W \lb^k$ is used in the primal update in \eqref{eq:up_2nd}. Therefore, we do not need to compute $\lb$ accurately, but rather focus on approximating $W\lb^k$ and $W\Delta\lb^k$ instead.
 In order to ensure that  $W \Delta\lb^k$  lies in the subspace $\text{range}(W)$, we introduce an additional communication round and use $W^2 \Delta \hat \lb^k = -W H^kW\xb^k$ as an estimator of $W \Delta \hat \lb^k$. We remark that such estimation is exact under a complete graph with $Z=\mathds{1}_n^\intercal \mathds{1}_n/n$, and it is more accurate if the underlying graph is a closer-to-complete one with all eigenvalues of $Z$ closer to either $1$ or $0$. Thus, by substituting the estimator $W^2\Delta \hat \lambda^k$, we obtain the dual update, 
\$
W \lb^{k+1} &= W \lb^{k} - \beta_2W^2 \Delta \hat \lb^k \\
& = W \lb^{k} + \beta_2W H^kW\xb^k.
\$
There are multiple equivalent $\lb$ solutions satisfying the above equation, all corresponding to the same primal update. One of these equivalent solutions can be obtained by omitting $W$ on both sides, which leads to the Newton-type dual update in \eqref{eq:up_2nd}. Thus, DISH-N with distributed implementation approximates the primal-dual Newton's method. Previous works \cite{tapia1977diagonalized, bertsekas2014constrained} have shown that the primal-dual Newton's method improves the convergence performance by utilizing second-order information. Thus, DISH-N  convergences efficiently when the approximations are good, i.e., with i.i.d. data distribution among agents for an empirical risk minimization problem and/or a closer-to-complete graph for the underlying topology.

\section{Theoretical Convergence}

In this section, we present the linear convergence rate for DISH in Algorithm \ref{alg:dish}, regardless of agents' choices of gradient-type or Newton-type updates. 

%In Section \ref{subsect:equiv_l_tL}, we show that using functions $L$ or ${L}$ in the analysis are indeed equivalent. 
%We remark that $\Delta_{{x}}^k = L(x^k, \lambda^k) - L(x^*(\lambda^k), \lambda^k)$ by its definition in \eqref{eq:terrors}. For convenience, we use the function $L$ in replace of ${L}$ hereafter.
\vskip4pt
\noindent\textbf{Properties of Primal and Dual Functions.} We first show some pivotal properties of the primal function $ L(\cdot,\lb)$.

\begin{lemma}\label{lem:L_property}
Under Assumption \ref{ass:hessian}, for any $\lb\in\RR^{nd}$, $L(\cdot,\lb)$ is $s$-strongly convex and $\ell_L$-Lipschitz smooth with $\ell_L = \ell+2\mu$.
\end{lemma}

Now we show the properties of the dual function $g(\lb)$. In particular, we denote by $\Lambda^{\textsf{OPT}}$ the dual optimal set to Problem \ref{prob:dual}, that is, for any $\lb^*\in\Lambda^{\textsf{OPT}}$, we have $g(\lb^*) = \max_{\lb} g(\lb)$.
The next lemma shows the properties of $g$.
\begin{lemma}\label{lem:g}
Under Assumption \ref{ass:hessian}, the dual function $g(\cdot)$ is $\ell_g$-Lipschitz smooth and it satisfies the PL inequality  \cite{karimi2016linear} with $\lb^*\in\Lambda^{\textsf{OPT}}$ that
\$
g(\lb^*) - g(\lb) \le \frac{1}{2p_g}\|\nabla g(\lb)\|^2,
\$
where constants $p_g = (1-\gamma)/(\ell + 2\mu)$ and $\ell_g =4/s$.
\end{lemma}

\vskip4pt
\noindent\textbf{Merit Function.} 
Now we introduce the merit function used in the analysis. We first define two performance metrics, the dual optimality gap and the primal tracking error, as follows, 
\#\label{eq:terrors}
&\Delta_{\lb}^k = g(\lb^*) - g(\lb^k), \notag\\
&\Delta_{\xb}^k = L({\xb}^k,\lb^k) - {L}({\xb}^*(\lb^k),\lb^k),
\#
where $x^*(\lb)$ is defined in \eqref{eq:x_star} and  $\lb^*\in\Lambda^{\textsf{OPT}}$ is a dual optimal point. We remark that $\Delta_{\xb}^k$ and $\Delta_{\lb}^k$ are both nonnegative by definition. Now we define a merit function to be used in the analysis by combing the performance metrics in \eqref{eq:terrors} as
\#\label{eq:merit}
\Delta^k = 9\Delta_{\lb}^{k} + \Delta_{\xb}^{k}.
\#
We remark that $\Delta^k\ge0$ for all $k\ge 0$.
We define $ \xb^{\textsf{OPT}}$ as the optimal solution of Problem \ref{prob:constrained}. The strong duality implies that $\xb^*(\lb^*) =  \xb^{\textsf{OPT}}$ for any $\lb^*\in\Lambda^{\textsf{OPT}}$. We will show that using DISH, $\Delta^k$ converges to zero at a linear rate in Theorem \ref{thm:strong} and therefore, the primal sequence $\xb^k$ goes to the exact optimal solution $\xb^{\textsf{OPT}}$ linearly in Corollary \ref{coro:dish}.

\iffalse
The next lemma justifies the merit function $\Delta^k$ defined in \eqref{eq:merit}.
\begin{lemma}\label{lem:metrit}
The convergence of $\{\Delta^k\}$, defined in \eqref{eq:merit}, to zero guarantees the primal sequence $\{\xb^k\}$ and the dual $\{W\lb^k\}$ converge to the optimal points $\xb^{\textsf{OPT}}$ and $\bm{\zeta}^{\textsf{OPT}}$, respectively. 
\end{lemma}

{\red do we really have linear convergence of $x^k$??????? please formalize this part.., i.e., $\|x^k - x^*\| \leq \sqrt{\Delta^k}$, etc} \xn{oh yeah, you are right}
\fi

\vskip4pt
\noindent\textbf{Linear Convergence of DISH. } Now we present the theoretical linear convergence of DISH. We first define constants $0<\up_i\le\op_i$ and $0<\uq_i\le\oq_i$ as bounds on the positive definite local update matrices $P_i^k$ and $Q_i^k$ for $i\in\cN$,
\#\label{eq:const_p_q}
\up_i I_d \preceq P_i^k \preceq \op_i I_d, \quad \uq_i I_d \preceq Q_i^k \preceq \oq_i I_d.
\#
Specifically, with the options of gradient-type and Newton-type updates in \eqref{eq:primal_p_choice} and \eqref{eq:dual_q_choice}, we have $\up_i = \min\{1, 1/(\ell_i+\mu)\}$, $\op_i = \max\{1, 1/(m_i+\mu)\}$, $\uq_i = \min\{1, m_i+\mu\}$, and $\oq_i = \max\{1, \ell_i+\mu\}$. For convenience, we also define constants $\underline{\alpha}$ and $\underline{\beta}$ as lower bounds to eigenvalues of matrices $AP^k$ and $BQ^k$, respectively, as follows, 
\#\label{eq:under_a_b}
&\underline{\alpha} = \min_{i\in\cN}\{a_i\up_i\}\le\min_{i\in\cN}\{a_i\theta_{\min}(P^k_i)\} = \theta_{\min}(AP^k),\notag\\
&\underline{\beta} = \min_{i\in\cN}\{b_i\uq_i\}\le\min_{i\in\cN}\{b_i\theta_{\min}(Q_i^k)\} = \theta_{\min}(BQ^k).
\#

Now we show the main theorem stating that DISH in Algorithm \ref{alg:dish} convergence linearly in terms of $\Delta^k$. 
\begin{theorem}[Linear Convergence of DISH]\label{thm:strong}
For any given $\mu \ge 0$, under Assumption \ref{ass:hessian}, we suppose that the stepsizes $\{a_i, b_i\}_{i\in \cN}$ satisfy the following conditions,
\#
&0 < a_i\le1/[2\op_i(s/16 + \ell + 2\mu)],\notag\\
&0 < b_i\le\min\{s/64, \underline{\alpha}s^2/60\} /\oq_i,  \label{eq:steps}
\#
where $\op_i$ and $\oq_i$ are defined in \eqref{eq:const_p_q} and $\underline{\alpha}$ is defined in \eqref{eq:under_a_b}. Then for all $k = 0,1,\cdots, K-1$, the iterates generated from DISH in Algorithm \ref{alg:dish} satisfy
\$
\Delta^{k+1} \le \rbr{1-\rho}\Delta^k,
\$
where $\rho = \min \{(1-\gamma)\underline{\beta}/[9(\ell+4\mu)], s\underline{\alpha}/2 \}$ with $\underline{\beta}$ defined in \eqref{eq:under_a_b} and $\Delta^k$ is the merit function defined in \eqref{eq:merit}. 
\end{theorem}

\begin{proof}[Proof Sketch of Theorem \ref{thm:strong}]
Now we sketch the proof of Theorem \ref{thm:strong}. Due to the coupled nature of primal and dual updates in DISH in \eqref{eq:up}, our main idea for analyzing the primal-dual framework is to bound the dual optimality gap $\Delta_{\lb}^k$ and the primal tracking error $\Delta_{\xb}^k$ through coupled inequalities. We decompose our analysis into three steps. 

\vskip2pt
\noindent\textit{Step 1: Bounding the Dual Optimality Gap $\Delta_{\lb}^{k+1}$.} We first bound the updated dual optimality gap $\Delta_{\lb}^{k+1}$ with an alternative primal tracking error $\|\nabla_{\xb}L({\xb}^k,\lb^k)\|^2$ using the Lipschitz smoothness of the dual function $g$. For convenience, we define a constant $\beta$ as an upper bound of $\|BQ^k\|$ as follows,
\# \label{eq:beta}
\beta = \max_{i\in\cN}\{b_i\oq_i\}\ge \max_{i\in\cN}\{b_i\|Q_i^k\|\} = \|BQ^k\|.
\#
The following proposition shows the obtained inequality.
\begin{proposition}\label{prop:gL}
Under Assumption \ref{ass:hessian}, given a constant $\mu>0$ and stepsizes $\{a_i,b_i\}_{i\in\cN}>0$, for all $k = 0,1,\cdots,K-1$, the  iterates  generated  from DISH in Algorithm \ref{alg:dish} satisfy
\$
\Delta_{\lb}^{k+1} &\le \Delta_{\lb}^{k} - \rbr{\frac{1}{2}-\beta\ell_g}\norm{\nabla g(\lb^{k})}^2_{BQ^k} \\
& \qquad + \rbr{\frac{1}{2}+\beta\ell_g}\frac{4\beta}{s^2} \norm{\nabla_{\xb}L({\xb}^k,\lb^k)}^2,
\$
where $\beta$ is defined in \eqref{eq:beta} and $\ell_g$ is defined in Lemma \ref{lem:g}.
\end{proposition}

\vskip2pt
\noindent\textit{Step 2: Bounding the Primal Tracking Error $\Delta_{\xb}^{k+1}$.} Next, we derive a bound of the updated primal tracking error $\Delta_{\xb}^{k+1}$ by an alternative dual optimality gap $\nabla g(\lb^k)$. We use the Lipschitz smoothness of $L(\cdot,\lb)$ to show the following result.
\begin{proposition}\label{prop:Lg}
Under Assumption \ref{ass:hessian}, given a constant $\mu>0$ and stepsizes $\{a_i,b_i\}_{i\in\cN}>0$, for all $k = 0,1,\cdots,K-1$, the  iterates  generated  from DISH in Algorithm \ref{alg:dish} satisfy
\$
\Delta_{\xb}^{k+1} & \le  \Delta_{\xb}^k  + 3\|\nabla g(\lb^k)\|^2_{BQ^k} - \|\nabla_{\xb} L(\xb^{k}, \lb^{k})\|_{D^k}^2 \\
&\quad+ \Delta_{\lb}^k - \Delta_{\lb}^{k+1}, 
\$
where matrix $D^k = AP^k - (2\beta+{\ell_L}/{2})A^2(P^k)^{2}-12\beta/ {s^2}I_{nd}$ with $\beta$ and $\ell_L$ defined in \eqref{eq:beta} and Lemma \ref{lem:L_property}, respectively.
\end{proposition} 
% See Section \ref{sect:a} for a detailed proof of Lemma \ref{prop:Lg}. 
\vskip2pt
\noindent\textit{Step 3: Putting Things Together.}
Finally, we take a linear combination of the coupled inequalities in Propositions \ref{prop:gL} and \ref{prop:Lg}. We use the strong convexity of $L(\cdot, \lb)$ and the PL inequality satisfied by $g(\lb)$ in Lemmas \ref{lem:L_property} and \ref{lem:g}. By some  algebraic manipulations, when the stepsizes satisfy \eqref{eq:steps}, we prove the linear convergence of DISH in Theorem \ref{thm:strong}.
\end{proof} 

The following corollary shows that DISH finds the optimal solution $\xb^{\textsf{OPT}}$ at a linear rate.
\begin{corollary}\label{coro:dish}
Under Assumption \ref{ass:hessian}, when the stepsizes $\{a_i, b_i\}_{i\in \cN}$ satisfy conditions in \eqref{eq:steps}, for all $k = 0,\cdots, K-1$, the iterates generated from DISH in Algorithm \ref{alg:dish} satisfy
\$
\|\xb^k - \xb^{\textsf{OPT}}\|^2 \le c(1-\rho)^k,
\$
where the constant $c = 4\ell_L\Delta^0/[s\cdot\min\{\ell_L,9s\}]$ with $\ell_L$ defined in Lemma \ref{lem:L_property}. 
\end{corollary}

Theorem \ref{thm:strong} shows a linear (Q-linear) convergence rate of DISH in terms of the merit function $\Delta^k$, regardless of agents' choices of gradient-type or Newton-type updates. Corollary \ref{coro:dish} guarantees that DISH converges linearly to the exact optimal solution $\xb^{\textsf{OPT}}$. We can also show that the dual sequence goes to the optimal point. We remark that since only constants $\up_i$, $\op_i$, $\uq_i$, $\oq_i$ defined in \eqref{eq:const_p_q} are needed in Theorem \ref{thm:strong}, agents can adopt any local updates as long as the update matrices are positive definite with uniform upper and lower bounds.

The provable linear rate of DISH recovers the linear rate of some existing distributed methods with exact convergence. Such methods include gradient-type methods such as EXTRA \cite{shi2015extra}, DIGing \cite{nedic2017achieving}, and \cite{qu2017harnessing}  and other methods that adopt Newton or quasi-Newton information like PD-QN \cite{eisen2019primal} and ESOM \cite{mokhtari2016decentralized} as discussed in Section \ref{sec:other_methods}. We remark that when all agents take Newton-type updates in both primal and dual spaces,  DISH does not give faster than linear rate. This is due to the distributed approximations made in both primal and dual Newton steps.

The linear rate $1-\rho$ in Theorem \ref{thm:strong} depends on the network structure $\gamma$, objective function properties $\ell$ and $s$, the augmentation penalty $\mu$, and the worst case of update matrices $\underline{\alpha}$ and $\underline{\beta}$. Although the theorem is conservative relying on the worst agents' updates,
as numerical experiments will show in Section \ref{sect:experiments}, DISH can achieve faster performance when more agents adopt Newton-type updates since the local information is more fully utilized.

%!TEX root = ieee_tsp/main.tex

\section{Numerical Experiments}\label{sect:experiments}
In this section, we present numerical studies of DISH on convex distributed  empirical risk minimization problems including linear least squares and binary classifications. All the experiments are conducted on 3.30GHz Intel Core i9 CPUs, Ubuntu 20.04.2, in Python 3.8.5. Our code is publicly available at \href{https://github.com/xiaochunniu/DISH}{https://github.com/xiaochunniu/DISH}.

\vskip4pt
\noindent\textbf{Experimental Setups.}
We evaluate all methods on two setups, both with synthetic data. In each setup, the underlying network is randomly generated by the Erd\H{o}s-R\'{e}nyi model with $n$ nodes (agents) and probability $p$ to generate each edge. We denote by $\delta_i$ the degree of node $i$ and $\delta_{\max} = \max_{i\in\cN}\{\delta_i\}$ the largest degree of the network. We define the elements of the consensus matrix $Z$ as $z_{ij} = 1/(\delta_{\max}+1)$ for $\{i,j\}\in\cE$, $z_{ii} = 1-\delta_i/(\delta_{\max}+1)$ for $i\in\cN$, and $z_{ij}=0$ otherwise. In each setup, the decision variable is $d$-dimensional and there are total amount $N=\sum_{i\in\cN}N_i$ of data in the network with the local dataset size $N_i$ at agent $i$. Here are more details of the setups. 

\vskip2pt
\noindent\textit{Setup 1: Decentralized Linear Least Squares over a Random Graph.} We first consider the decentralized regularized linear least squares problem as follows,
\$
\min_{\omega\in\RR^d} \frac{1}{2N}\sum_{i=1}^n  \norm{A_{i} \omega - y_{i}}^2 + \frac{\rho}{2}\norm{\omega}^2,
\$
where $A_{i}\in \RR^{N_i\times d}$ and $y_{i}\in \RR^{N_i}$ are the feature matrix and the response vector at agent $i$, respectively, and $\rho\ge0$ is the penalty parameter. Specifically, we set $n=10$, $p=0.7$, $d=5$, $N_i= 50$ for all $i\in\cN$, and $\rho=1$. We generate matrices $\hat A_i\in\RR^{50\times 5}$, noise vectors $v_i\in\RR^{50}$ for $i\in\cN$, and a vector $\omega_0\in\RR^{5}$ from standard Normal distributions. We set feature matrices $A_i=\hat A_i \Theta$, where $\Theta=\diag\{10, 10, 0.1, 0.1, 0.1\}\in\RR^{5\times 5}$ is the scaling matrix. 
We generate the response vector $y_i\in\RR^{50}$ by the formula
$
y_i = A_i \omega_0 + v_i$ for $i\in\cN$.

\vskip2pt
\noindent\textit{Setup 2: Decentralized Logistic Regression over a Random Graph.} 
The second setup studies the regularized logistic regression model for solving binary classification problems
\$
\min_{\omega\in\RR^d}\frac{1}{N} \sum_{i=1}^n  \big[- y_i^\intercal \log h_i - (1-y_i)^\intercal\log(1-h_i)\big] + \frac{\rho}{2}\|\omega\|^2,
\$
where $h_i= 1/(1+\exp(-A_i\omega))$, $A_{i}\in \RR^{N_i\times d}$, and $y_{i}\in \RR^{N_i}$ are the known feature matrix and label vector at agent $i$, respectively, and $\rho$ is the penalty parameter.  
Specifically, we set $n=20$, $p=0.5$, $d=3$, $N_i= 50$ for all $i\in\cN$, and $\rho=1$. 
We generate matrices $\hat A_i\in\RR^{50\times 3}$, noise vectors $v_i\in\RR^{50}$ for $i\in\cN$, and a vector $\omega_0\in\RR^{3}$ from Normal distributions. We scale $\hat A_i$ with matrix $\Theta=\diag\{10, 0.1, 0.1\}\in\RR^{3\times 3}$ and set feature matrices to be $A_i=\hat A_i \Theta$.
The response vector $y_i\in\RR^{50}$ is generated by $y_i = \argmax(\text{softmax}(A_i \omega_0 + v_i))$.
The generated underlying networks are shown in Figure \ref{fig:ER}.

\begin{figure}[!t]
\centering
  \subfloat[Least Squares in Setup 1]{%
       \includegraphics[width=0.49\linewidth]{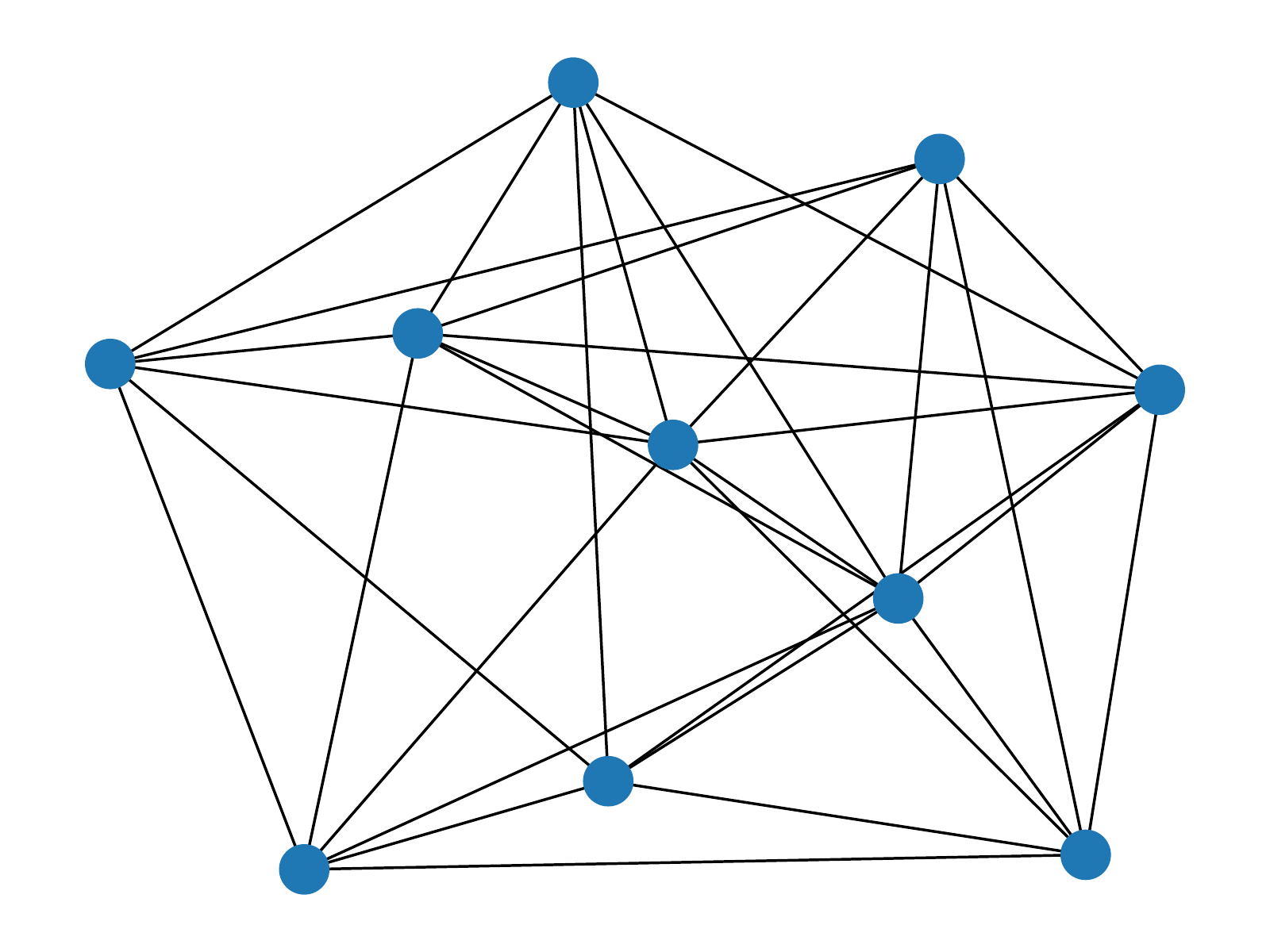}}
    \hfill
  \subfloat[Logistic Regression in Setup 2]{%
        \includegraphics[width=0.49\linewidth]{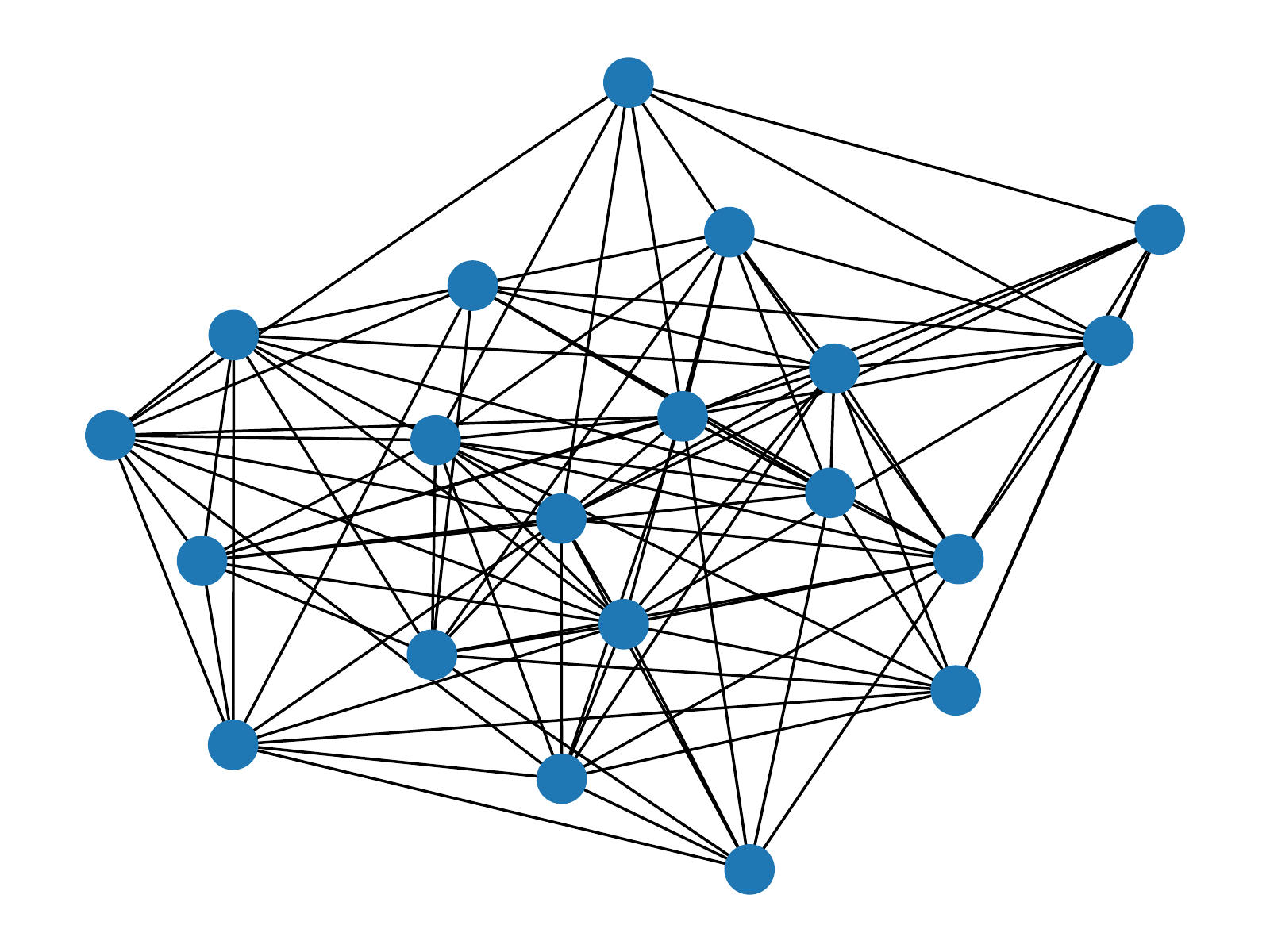}}
     \caption{Underlying Networks.}
     \label{fig:ER}
\end{figure}

\vskip4pt
\noindent\textbf{Implemented Methods.} We implement EXTRA \cite{shi2015extra}, ESOM-$0$ \cite{mokhtari2016decentralized}, and DISH in Algorithm \ref{alg:dish} on the introduced two setups. For convenience, we denote by DISH-$K$ DISH with $K$ agents performing Newton-type updates and the other agents performing gradient-type updates all the time. Moreover, we represent DISH-G\&N as DISH with all agents switching between gradient-type and Newton-type updates once in a while. In particular, for DISH-G\&N-U and DISH-G\&N-LN, we generate $t_i\sim U[5,50]$ and $t_i\sim\text{lognormal}(2,4)+30$, respectively. In both cases, we let agent $i$ change its updates type every $t_i$ iterations with the initial updates uniformly sampled from $\{$`gradient-type', `Newton-type'$\}$.
We remark that all these methods require one communication round with the same communication costs for each iteration independent of the update type.

For all setups and methods, we tune stepsizes and parameters by grid search in the range $[2^{-6}, 2^4]$ and select the optimal ones that minimize the number of iterations to reach a predetermined relative error threshold, measured by $\|\xb^k-\xb^{\textsf{OPT}} \|/\|\xb^0-\xb^{\textsf{OPT}}\|$, where the optimal point $\xb^{\textsf{OPT}}$ is obtained by a centralized solver for Problem \ref{prob:original}. We remark that in DISH we fix $a_i=1$ when agent $i$ takes Newton-type updates to mimic primal Newton's step.

\vskip4pt
\noindent\textbf{Results and Conclusions.} In both Figures \ref{fig:dish_1} and \ref{fig:dish_2}, the $x$-axis shows the number of communication rounds (iterations) and the $y$-axis is the logarithm of the relative error. As shown in Figures \ref{fig:dish_1} and \ref{fig:dish_2}, it is clear that the DISH framework has a linear convergence performance regardless of agents' choice of gradient-type and Newton-type updates, which validates the theoretical guarantees in Theorem \ref{thm:strong}.
\begin{figure}[!t]
\centering
  \subfloat[Least Squares in Setup 1]{%
       \includegraphics[width=0.49\linewidth]{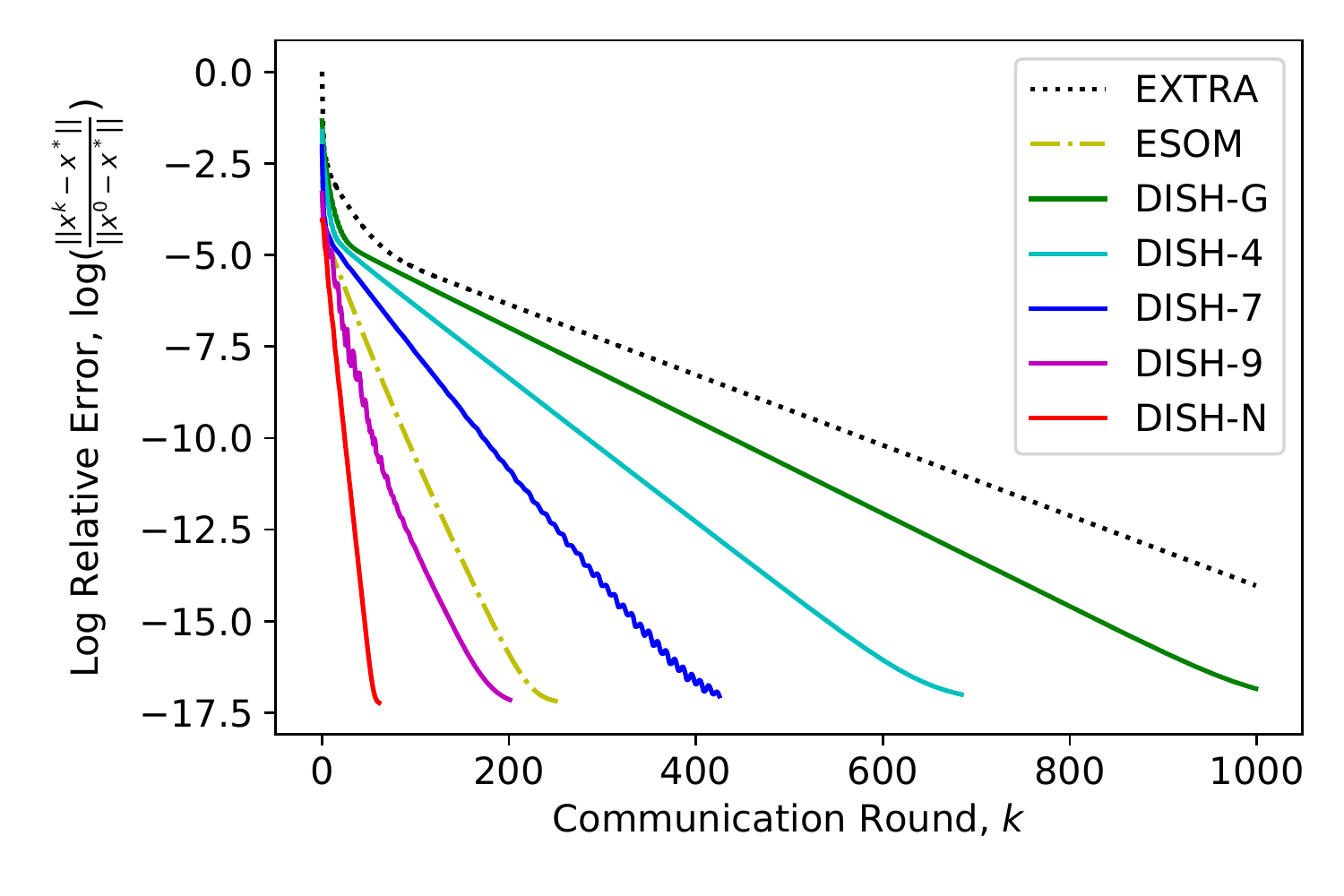}}
    \hfill
  \subfloat[Logistic Regression in Setup 2]{%
        \includegraphics[width=0.49\linewidth]{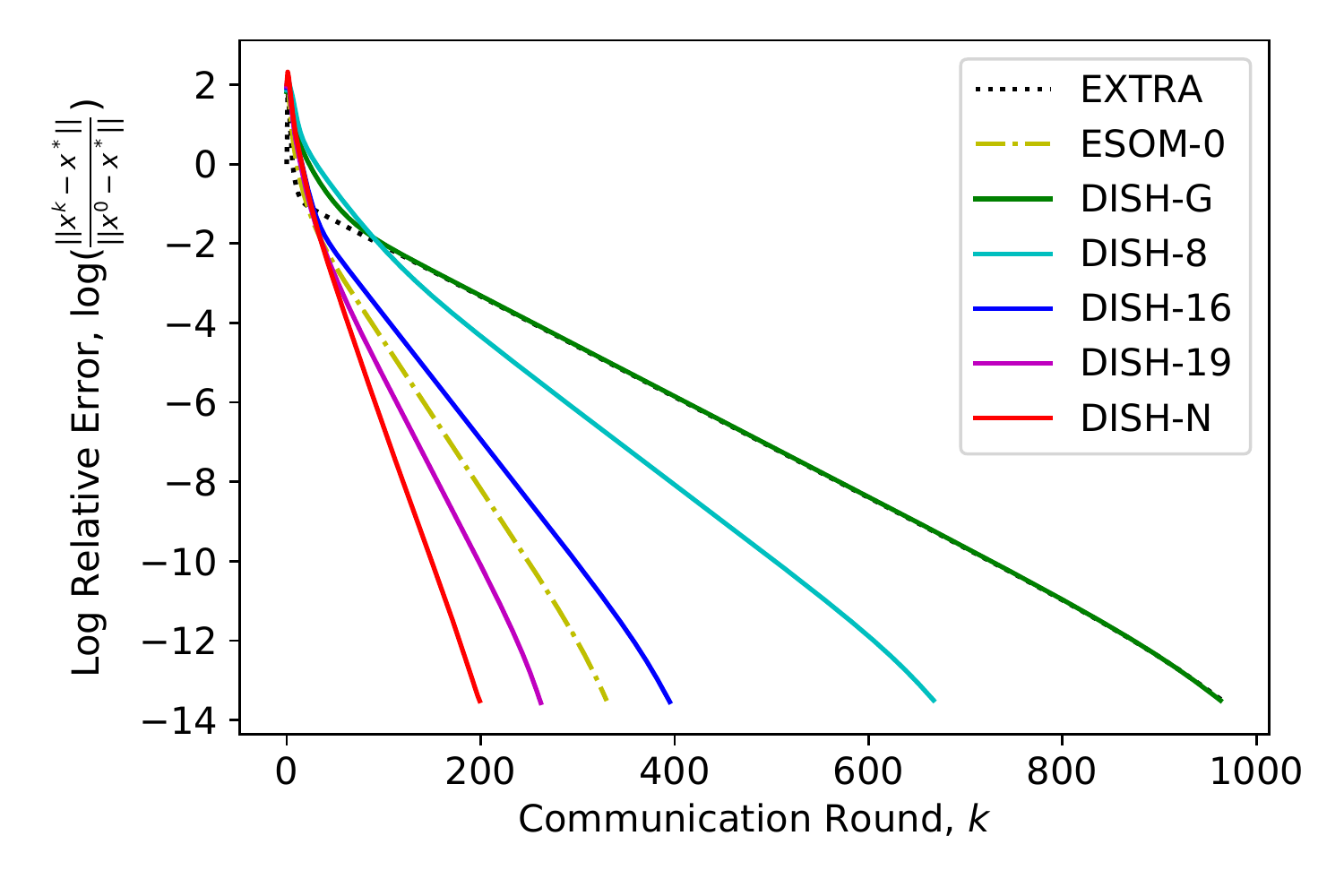}}
     \caption{Performance of EXTRA, ESOM-$0$, and DISH.}
     \label{fig:dish_1}
\end{figure}

\begin{figure}[!t]
\centering
  \subfloat[Least Squares in Setup 1]{%
       \includegraphics[width=0.49\linewidth]{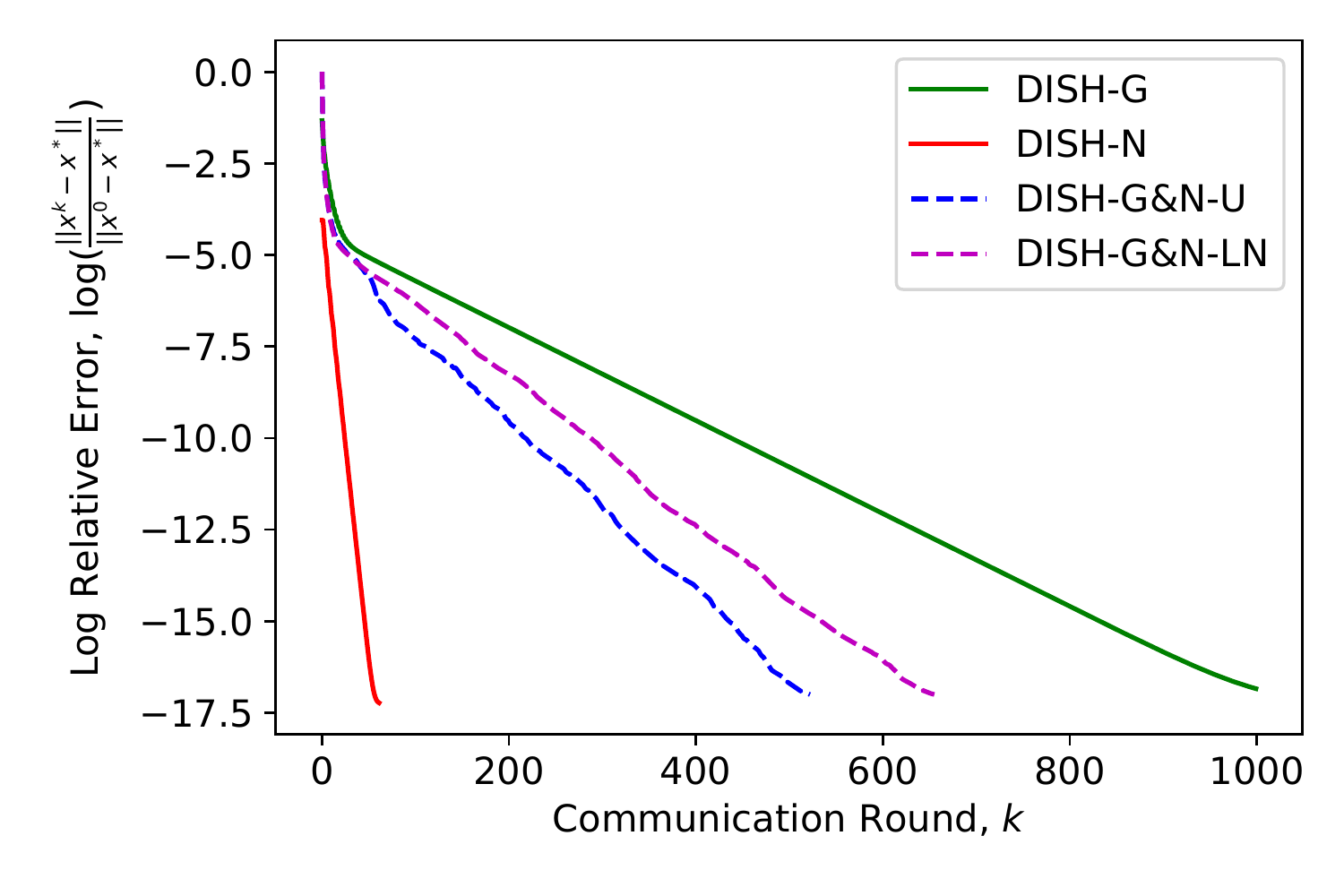}}
    \hfill
  \subfloat[Logistic Regression in Setup 2]{%
        \includegraphics[width=0.49\linewidth]{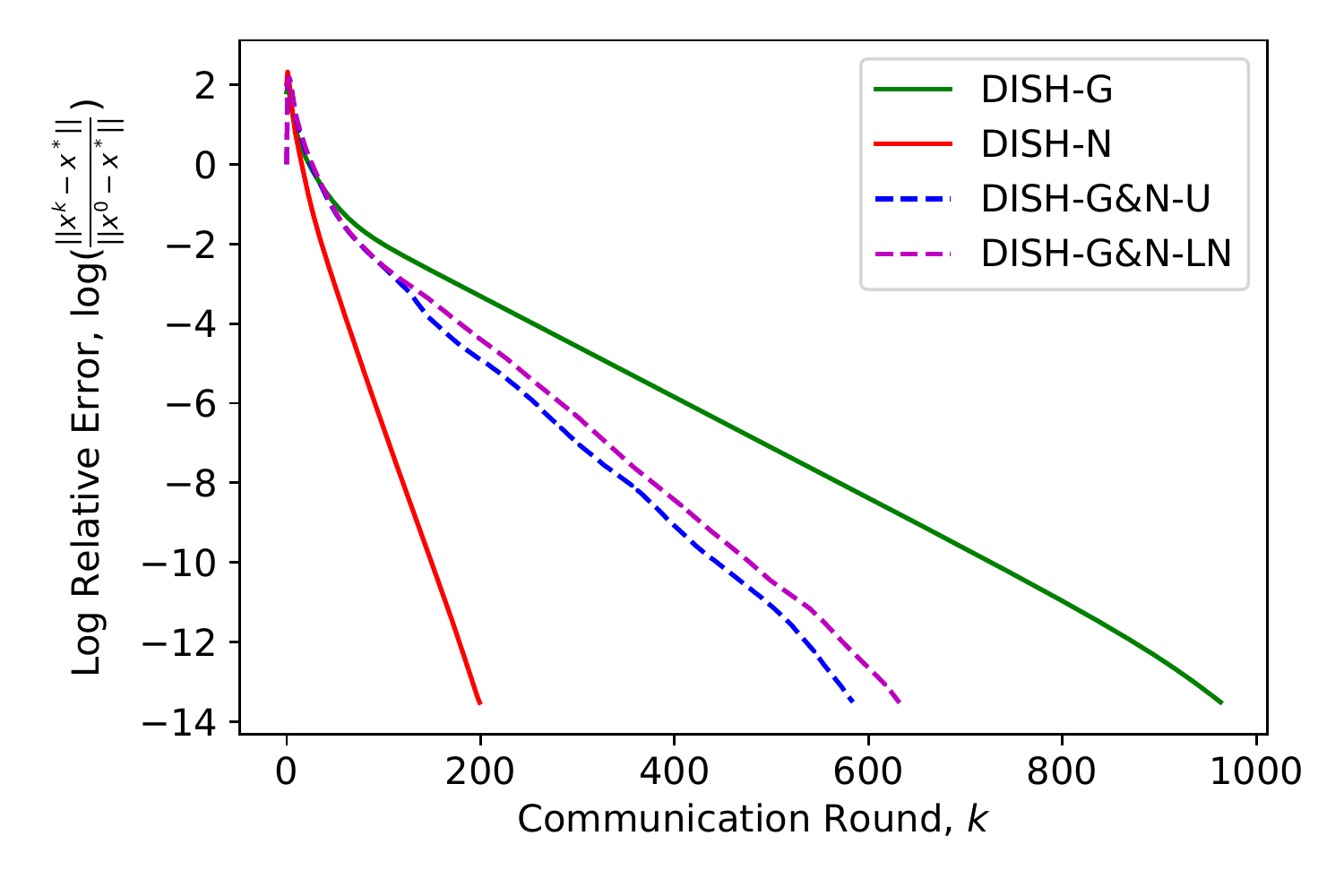}}
     \caption{Performance of DISH-G\&N.}
     \label{fig:dish_2}
\end{figure}

As shown in Figure \ref{fig:dish_1}, primal-dual gradient-type methods, EXTRA and DISH-G, perform similarly due to their similar update formulas. However, when some agents take Newton-type updates, DISH improves the overall training speed and outperforms the baseline method DISH-G consistently. In particular, the second-order DISH-N method outperforms ESOM-$0$ in many scenarios, implying DISH-N benefits from the dual Hessian approximation. 

Moreover, as the number of agents that perform Newton-type updates, $K$, increases, the numerical convergence of DISH is likely to become faster since the Hessian information can be more fully utilized. This observation suggests that in practical systems, those agents with higher computational capabilities and/or cheaper costs to perform computation can choose to take Newton-type updates locally to help speed up the overall convergence of the whole system.

In traditional distributed optimization algorithms, all agents perform the same type of updates. The complexity of the method is determined by the agents equipped with the worst computation hardware. While in DISH, since efficient Newton-type updates are involved at parts of the network, the overall system enjoys a faster convergence speed compared to systems running gradient-type methods only. Therefore, we can maximally leverage the parallel heterogeneous computation capabilities in this setting.

\section{Final Remarks and Future Work}
This paper proposes DISH as a distributed hybrid primal-dual algorithmic framework allowing agents to perform either gradient-type or Newton-type updates based on their computation capacities. We show a linear convergence rate of DISH for strongly convex functions. Numerical studies are provided to demonstrate the efficacy of DISH in practice. 

We highlight a few interesting directions for future works on the DISH framework and distributed optimization. First, we expect DISH to be generalized to broader settings like time-varying graphs or systems with non-convex objective functions. Also, we could involve stochastic methods in DISH. For instance, agents could perform stochastic gradient-type or subsampled Newton-type methods locally. Moreover, we could consider asynchronous updates in DISH, where at each communication round, only a randomly selected subset of the agents take computation steps since it is possible that only a few agents are active in practice.

\bibliographystyle{IEEEtran}
\bibliography{ref}

\appendices
\section{Proof of Lemmas}
\subsection{Proof of Lemma \ref{lem:dual_hessian_app}}
\begin{proof}
By the definition of  $\hat\nabla g(\lb^k)$ and $\hat\nabla^2 g(\lb^k)$, the dual Newton's step in \eqref{eq:up_2nd_d} is as follows, 
\$
-W\rbr{\nabla_{\xb\xb}^2L(\xb^k,\lb^k)}^{-1}W \Delta\lb^k =  W\xb^k.
\$
We note that the null space of the matrix $W$ is $\text{null}(W) = \{\mathds{1}_n\otimes y: y\in\RR^d\}$. Thus, there exists $y^k\in\RR^d$ such that
\$
\rbr{\nabla_{\xb\xb}^2L(\xb^k,\lb^k)}^{-1}W \Delta\hat\lb^k + \xb^k = \mathds{1}_n\otimes y^k.
\$
Rearranging terms in the previous equation, we have 
\#\label{eq:wlambda2}
W \Delta\hat\lb^k & = \nabla_{\xb\xb}^2L(\xb^k,\lb^k)\rbr{\mathds{1}_{n}\otimes y^k - {\xb}^k} \notag \\
& = \rbr{\nabla^2 f({\xb}^k) + \mu W}\rbr{\mathds{1}_{n}\otimes y^k - {\xb}^k}.
\# 
Since $(\mathds{1}_{n}^\intercal\otimes I_d)W = 0$, by multiplying $\mathds{1}_{n}^\intercal\otimes I_d$ on both sides of \eqref{eq:wlambda2}, we have
\$
0 = \rbr{\mathds{1}_{n}^\intercal\otimes I_d}\nabla^2 f(\xb^k)\rbr{\mathds{1}_{n}\otimes y^k - {\xb}^k}.
\$
Thus, we have $(\sum_{i\in\cN} \nabla^2 f_i(x_i^k))y^k =\sum_{i\in\cN} \nabla^2 f_i(x_i^k)x_i^k$. By multiplying the matrix inverse on both sides, we have
\$
y^k = \big(\sum_{i\in\cN}\nabla^2f_i(x_i^k) \big)^{-1}\sum_{i\in\cN}\nabla^2f_i(x_i^k) x_i^k.
\$
Substituting the preceding relation into \eqref{eq:wlambda2}, we conclude the proof of the lemma.
\end{proof}

\subsection{Proof of Lemmas \ref{lem:L_property} and \ref{lem:g}}

For convenience, we first define $\tilde f:\RR^{nd}\to\RR$ as $\tilde f(\xb) = f(\xb) + \mu \xb^{\intercal}W\xb / 2$ and show the following lemma.

\begin{lemma}\label{lem:penalized_f}
Under Assumption \ref{ass:hessian}, eigenvalues of the Hessian $\nabla^2 \tilde f(\cdot)$ are bounded by constants $s\le\ell_L$ as $s I_{nd}\le\nabla^2 \tilde f(\cdot) \le \ell_L I_{nd}$, where $\ell_L = \ell+2\mu$.
\end{lemma}
\begin{proof}
By the definition of $\tilde f$ as above, we have
\$\nabla^2 \tilde f(\xb) = \nabla^2f(\xb) + \mu W^2.\$ By the fact that $\rho(Z) = 1$, we have $\rho(W) = \rho(I_d-Z) \le 2$. Under Assumption \ref{ass:hessian}, using $0\preceq W\preceq 2I_{nd}$, we have
\$
sI_{nd}\preceq \nabla^2 \tilde f(\xb) \preceq (\ell+2\mu)I_{nd}.
\$
This concludes the proof of the Lemma.
\end{proof}

Now we prove Lemma \ref{lem:L_property}.
\begin{proof}[Proof of Lemma \ref{lem:L_property}]
We prove the lemma by providing both upper and lower bounds on the partial Hessian $\nabla_{\xb\xb}^2 L(\cdot,\lb)$.
Given any fixed $\lb\in\RR^{nd}$, by taking partial Hessian with respect to $\xb$ of the function $L$ defined in \eqref{eq:lagrangian_func}, we have
\$
\nabla^2_{\xb\xb} L(\xb,\lb) = \nabla^2f(\xb) + \mu W =\nabla^2 \tilde f(\xb).
\$ 
Thus, by Lemma \ref{lem:penalized_f}, with $\ell_L=\ell+2\mu$, we have
\$
sI_{nd}\preceq \nabla^2_{\xb\xb} L(\xb,\lb)=\nabla^2 \tilde f(\xb) \preceq \ell_L I_{nd}.
\$
This concludes the proof of the Lemma.
\end{proof}

Now we denote by $\tilde f^*$ the convex conjugate of $\tilde f$. We observe that $g(\lb) = \min_{{\xb}} \{\tilde f(\xb) + \lb^\intercal W\xb\} = - \tilde f^*(-W\lb)$.
The next lemma shows the properties of $\tilde f^*$.
\begin{lemma}\label{lem:penalized_f_conjugate}
Under Assumption \ref{ass:hessian}, $\tilde f^*$ is $s_{\tilde f^*}$-strongly convex and $\ell_{\tilde f^*}$-Lipschitz smooth with $s_{\tilde f^*} = 1/\ell_L$ and $\ell_{\tilde f^*} = 1/s$.
\end{lemma}
\begin{proof}
Under Assumption \ref{ass:hessian}, since $\tilde f$ is $s$-strongly convex and $\ell_L$-Lipschitz smooth implied by Lemma \ref{lem:penalized_f}, the conjugate $\tilde f^*$ is $s_{\tilde f^*}$-strongly convex and $\ell_{\tilde f^*}$-Lipschitz smooth with constants $s_{\tilde f^*} = 1/\ell_L$ and $\ell_{\tilde f^*} = 1/s$\cite{hiriart2004fundamentals}. 
\end{proof}

We denote by $-\bm{\zeta}^{\textsf{OPT}}$ the unique minimizer of $\tilde f^*$. Now we begin the proof of Lemma \ref{lem:g}.
\begin{proof}[Proof of Lemma \ref{lem:g}]
By the definition of $g$ in \eqref{prob:dual}, we have $g(\lb)= - \tilde f^*(-W\lb)$, where $\tilde f^*$ is the convex conjugate of $\tilde f$. Now we apply properties of $\tilde f^*$ in Lemma \ref{lem:penalized_f_conjugate}.
Since the conjugate $\tilde f^*$ is $s_{\tilde f^*}$-strongly convex, we denote by $-\bm{\zeta}^{\textsf{OPT}}$ its unique minimizer. We first consider the dual optimal set $\Lambda^{\textsf{OPT}}$. Since $\max g(\lb) = -\min\tilde f^*(-W\lb)$, for any $\lb^*\in\Lambda^{\textsf{OPT}}$, we have $g(\lb^*) =-\tilde f^*(-W\lb^*) =-\tilde f^*(-\bm{\zeta}^{\textsf{OPT}})$. Thus, we have
\$
\Lambda^{\textsf{OPT}} = \{\lb^*\colon W\lb^* = \bm{\zeta}^{\textsf{OPT}}\}.
\$

Now we consider the dual function $g(\lb)$. According to \cite{karimi2016linear}, since $g(\lb)= - \tilde f^*(-W\lb)$ and $\tilde f^*$ is $s_{\tilde f^*}$-strongly convex, the function $-g(\lb)$ satisfies the PL inequality that 
\$
g(\lb^*) - g(\lb) \le \frac{1}{2p_g}\|\nabla g(\lb)\|^2,
\$
where $p_g = \sigma_{\min}^+(-W) \cdot s_{\tilde f^*} = (1-\gamma)/(\ell + 2\mu)$.
As for the Lipschitz smoothness, for any $\lb_1,\lb_2\in\RR^{nd}$, by straightforward algebraic manipulations, we have
\$
\|\nabla g(\lb_1) -\nabla g(\lb_2)\| &= \|\nabla_{\lb} \tilde f^*(-W\lb_2) -\nabla_{\lb} \tilde f^*(-W\lb_1)\| \\
&= \|W[\nabla \tilde f^*(-W\lb_2) -\nabla \tilde f^*(-W\lb_1)] \| \\
&\le \ell_{\tilde f^*}\|W\| \|W\lb_1-W\lb_2\|\\
&\le \ell_{\tilde f^*}\rho(W)^2\|\lb_1 - \lb_2\| \\
&= (4/s)\|\lb_1 - \lb_2\|,
\$
where the first inequality is due to the $\ell_{\tilde f^*}$-Lipschitz smoothness of $\tilde f^*$. This concludes the proof of the Lemma.
\end{proof}

\section{Proof of Propositions}
Before presenting the analysis of Propositions \ref{prop:gL} and \ref{prop:Lg}, we first introduce a corollary of the strong convexity of $L(\cdot,\lb)$. 
\begin{lemma}\label{lem:mx}
Under Assumption \ref{ass:hessian}, for all $k = 0,1,\cdots,K-1$, the iterates generated from DISH in Algorithm \ref{alg:dish} satisfy
\$\norm{W{\xb}^k - W{\xb}^*(\lb^k)} \le \frac{2}{s} \norm{\nabla_{\xb}L(\xb^k,\lb^k)}.\$
\end{lemma}
\begin{proof}
Using the fact that $\rho(W)\le 2$, we have
\$
\|W{\xb}^k - W{\xb}^*(\lb^k)\| &\le\|W\| \|\xb^k - \xb^*(\lb^k)\|\notag \\
&\le 2\|\xb^k - \xb^*(\lb^k)\|.
\$
Using the $s$-strong convexity of $L(\cdot,\lb)$ in Lemma \ref{lem:L_property}, we bound the RHS in the above equation by $\|\nabla_{\xb}L({\xb}^k,\lb^k)\|$ as follows,
\$
\|\nabla_{\xb}L({\xb}^k,\lb^k)\| &= \|\nabla_{\xb}L({\xb}^k,\lb^k) - \nabla_{\xb}L({\xb}^*(\lb^k),\lb^k) \| \notag \\
&\ge s \|\xb^k - \xb^*(\lb^k)\|.
\$
Thus, by combing the preceding two equations, we have
\$
\|W{\xb}^k - W{\xb}^*(\lb^k)\| \le \frac{2}{s} \|\nabla_{\xb}L({\xb}^k,\lb^k)\|.
\$
This concludes the proof of the lemma. 
\end{proof}

We also provide the following lemma to bound the dual update $\|\lb^{k+1} - \lb^k\|$ by an alternative primal tracking error $\|\nabla_{\xb}L({\xb}^k,\lb^k)\|$ and the dual gradient $\|\nabla g(\lb^k)\|$.

\begin{lemma}\label{lem:gap_lambda}
Under Assumption \ref{ass:hessian}, for all $k = 0,1,\cdots,K-1$, the iterates generated from DISH in Algorithm \ref{alg:dish} satisfy
\$
&\|W{\xb}^{k}\|^2_{BQ^k}\le\frac{8\beta}{s^2} \|\nabla_{\xb}L({\xb}^k,\lb^k)\|^2 + 2\|\nabla g(\lb^k)\|^2_{BQ^k}, \\
&\|\lb ^{k+1} - \lb ^k\|^2\le \frac{8\beta^2}{s^2} \norm{\nabla_\xb  L({\xb  }^k,\lb ^k)}^2 + 2\beta\norm{\nabla g(\lb ^k)}^2_{BQ^k}.
\$
\end{lemma}
\begin{proof}
By the inequality that for any $a,b\in\RR^{nd}$, $2a^\intercal b \le \|a\|^2 + \|b\|^2$, we have
\#\label{eq:wx_bd}
&\norm{W{\xb  }^{k}}^2_{BQ^k} \notag\\
&\quad \le 2\norm{W{\xb  }^{k} - W{\xb  }^*(\lb ^k)}^2_{BQ^k} + 2\norm{W{\xb  }^*(\lb ^k)}^2_{BQ^k}  \notag\\
&\quad= 2\norm{W{\xb  }^{k} - W{\xb  }^*(\lb ^k)}^2_{BQ^k} + 2\norm{\nabla g(\lb ^k)}^2_{BQ^k} \notag\\
&\quad\le\frac{8\beta}{s^2} \norm{\nabla_\xb  L({\xb  }^k,\lb ^k)}^2 + 2\norm{\nabla g(\lb ^k)}^2_{BQ^k},
\#
where the equality follows from Lemma \ref{lem:dual_hessian} and the last inequality is due to \eqref{eq:beta} and Lemma \ref{lem:mx}.
Now by taking the dual update in \eqref{eq:up}, we obtain
\$
\|\lb ^{k+1} - \lb ^k\|^2 
&= \|BQ^kW{\xb  }^k\|^2 \\
& \le \beta\norm{W{\xb  }^{k}}^2_{BQ^k} \\
& \le \frac{8\beta^2}{s^2} \norm{\nabla_\xb  L({\xb  }^k,\lb ^k)}^2 + 2\beta\norm{\nabla g(\lb ^k)}^2_{BQ^k},
\$
where the first inequality is due to \eqref{eq:beta} and the last equality follows from \eqref{eq:wx_bd}. 
\end{proof}

Now we show the proof of Proposition \ref{prop:gL}.

\begin{proof}[Proof of Proposition \ref{prop:gL}]
Using the $\ell_g$-Lipschitz continuity of $\nabla g(\cdot)$ in Lemma \ref{lem:g}, we have \cite{nesterov2018lectures},
\#\label{eq:step1}
&g(\lb^{k+1}) \\
&\quad \ge g(\lb^k) + \inner{\nabla g(\lb^{k})}{\lb^{k+1} - \lb^k} - \frac{\ell_g}{2}\norm{\lb^{k+1} - \lb^k}^2 \notag \\
&\quad = g(\lb^k) + \inner{\nabla g(\lb^{k})}{BQ^kW{\xb}^k} - \frac{\ell_g}{2}\norm{\lb^{k+1} - \lb^k}^2,\notag
\#
where the equality is due to the dual update in \eqref{eq:up}.
Now we consider the second term in \eqref{eq:step1}. By adding and subtracting $\inner{\nabla g(\lb^{k})}{BQ^k\nabla g(\lb^{k})}$ in the inner product, we have 
\#\label{eq:lambdainner}
&\inner{\nabla g(\lb^{k})}{BQ^kW{\xb}^k} \notag\\
&\qquad = \inner{\nabla g(\lb^{k})}{BQ^kW{\xb}^k - BQ^k\nabla g(\lb^{k})} \notag\\
&\qquad \qquad + \inner{\nabla g(\lb^{k})}{BQ^k\nabla g(\lb^{k})} \notag \\
&\qquad \ge -\frac{1}{2}\norm{\nabla g(\lb^{k})}^2_{BQ^k} - \frac{1}{2}\norm{W{\xb}^k - \nabla g(\lb^{k})}^2_{BQ^k}\notag\\
&\qquad \qquad  + \norm{\nabla g(\lb^{k})}^2_{BQ^k} \notag\\
&\qquad = \frac{1}{2}\norm{\nabla g(\lb^{k})}^2_{BQ^k} - \frac{1}{2}\norm{W{\xb}^k - W{\xb}^*(\lb^k)}^2_{BQ^k},  \notag \\
&\qquad \ge \frac{1}{2}\norm{\nabla g(\lb^{k})}^2_{BQ^k} - \frac{2\beta}{s^2}\norm{\nabla_\xb L(\xb^k,\lb^k)}^2,
\#
where the first inequality follows from the inequality that for any $a,b\in\RR^{nd}$, $2a^\intercal b \le \|a\|^2 + \|b\|^2$, the last equality is due to Lemma \ref{lem:dual_hessian}, and the last inequality is due to Lemma \ref{lem:mx}.

By substituting \eqref{eq:lambdainner} and Lemma \ref{lem:gap_lambda} into \eqref{eq:step1}, we have 
\$
g(\lb ^{k+1}) &\ge g(\lb ^k) +\rbr{\frac{1}{2}-\beta\ell_g}\norm{\nabla g(\lb ^{k})}^2_{BQ^k} \notag\\
&\qquad - \rbr{\frac{1}{2}+\beta\ell_g}\frac{4\beta}{s^2} \norm{\nabla_\xb L({\xb }^k,\lb ^k)}^2.
\$
By subtracting the dual optimal $g(\lb ^*)$ and taking negative signs on both sides of the above relation, we finally have
\$
\Delta_{\lb }^{k+1} & \le \Delta_{\lb }^{k} - \rbr{\frac{1}{2}-\beta\ell_g}\norm{\nabla g(\lb ^{k})}^2_{BQ^k} \notag\\
& \qquad + \rbr{\frac{1}{2}+\beta\ell_g}\frac{4\beta}{s^2} \norm{\nabla_\xb L({\xb }^k,\lb ^k)}^2,
\$
where $\Delta_{\lb}^k$ is defined in \eqref{eq:terrors}.
\end{proof}

Next, we prove Proposition \ref{prop:Lg}.

\begin{proof}[Proof of Proposition \ref{prop:Lg}]
By the definition in \eqref{eq:terrors}, the tracking error of the primal update $\Delta_{\xb}^{k+1}$ consists of the following three terms,
\#\label{eq:deltaL}
\Delta_{\xb}^{k+1} &=  L({\xb}^{k+1}, \lb^{k+1}) - L(\xb^*(\lb^{k+1}), \lb^{k+1}) \notag \\
&= \underbrace{L({\xb}^{k+1}, \lb^{k+1}) -L(\xb^{k+1}, \lb^{k})}_{\textstyle \text{term (A)}} \notag\\
& \qquad + \underbrace{L(\xb^{k+1}, \lb^{k}) -L(\xb^*(\lb^{k}), \lb^{k})}_{\textstyle \text{term (B)}}  \notag\\
& \qquad + \underbrace{L(\xb^*(\lb^{k}), \lb^{k}) - L(\xb^*(\lb^{k+1}), \lb^{k+1})}_{\textstyle \text{term (C)}},
\#
where term (A) measures the increase due to the dual update, term (B) represents the updated primal tracking error, and term (C) shows the difference between dual optimality gaps.

In the sequel, we upper bound terms (A)-(C), respectively.

\textit{Term (A).} By the definition of $L$ in \eqref{eq:lagrangian_func}, we have
\#\label{eq:termA}
&L({\xb}^{k+1}, \lb^{k+1}) -L(\xb^{k+1}, \lb^{k}) \notag\\ &\quad=(\lb^{k+1}-\lb^{k})^\intercal W{\xb}^{k+1} \notag\\
&\quad = \rbr{BQ^kW{\xb}^{k}}^\intercal W{\xb}^{k+1} \notag\\
&\quad = \|W\xb^k\|^2_{BQ^k} + \rbr{W\xb^k}^\intercal BQ^k\rbr{W\xb^{k+1} - W{\xb}^{k}} \notag\\
& \quad \le \frac{3}{2}\|W\xb^k\|_{BQ^k}^2 + \frac{1}{2}\|W\xb^{k+1} - W\xb^k\|_{BQ^k}^2, \#
where the inequality follows from the inequality that for any $a,b\in\RR^{nd}$, $2a^\intercal b \le \|a\|^2 + \|b\|^2$. 
We note that the first term in \eqref{eq:termA} can be upper bounded by Lemma \ref{lem:gap_lambda}. Now we upper bound the second term as follows,
\#\label{eq:a1ii}
\|W{\xb}^{k+1} - W{\xb}^{k}\|_{BQ^k}^2
& \le \|BQ^k\| \|W\|^2 \|\xb^{k+1} - \xb^k\|^2 \notag\\
& \le 4\beta \|\xb^{k+1} - \xb^k\|^2 \notag \\
& = 4\beta \|AP^k\nabla_\xb L(\xb^k,\lb^k)\|^2,
\#
where the second inequality follows from \eqref{eq:beta} and $\rho(W) \le 2$ and the equality is due to the primal update in \eqref{eq:up_pre}. 

By substituting Lemma \ref{lem:gap_lambda} and \eqref{eq:a1ii} into \eqref{eq:termA}, we have
\#\label{eq:term_a_1_result}
&L({\xb}^{k+1}, \lb^{k+1}) -L(\xb^{k+1}, \lb^{k}) \notag \\
&\le  \frac{12\beta}{s^2}\|\nabla_{\xb}L({\xb}^k,\lb^k)\|^2 + 3\|\nabla g(\lb^k)\|^2_{BQ^k} \notag \\
&\quad+ 2\beta\|AP^k\nabla_{\xb}L(\xb^k,\lb^k)\|^2.
\#
This provides an upper bound on term (A).

\textit{Term (B).} Using the $\ell_L$-Lipschitz continuity of $\nabla_\xb L(\xb ,\lb^k)$ from Lemma \ref{lem:L_property}, we have
\$
L(\xb ^{k+1}, \lb^{k}) &\le L(\xb ^{k}, \lb^{k}) + \nabla_\xb  L(\xb ^{k}, \lb^{k})^\intercal\rbr{\xb ^{k+1} - \xb ^k} \notag\\
&\quad + \frac{\ell_L}{2}\|\xb ^{k+1} - \xb ^k\|^2 \\
&= L(\xb ^{k}, \lb^{k}) + \frac{\ell_L}{2}\|AP^k\nabla_\xb L(\xb ^k,\lb^k)\|^2 \notag \\
&\quad - \nabla_\xb  L(\xb ^{k}, \lb^{k})^\intercal A P^k\nabla_\xb L(\xb ^k,\lb^k), \notag
\$
where the equality follows from the primal update in \eqref{eq:up_pre}. Subtracting $L(\xb ^{*}(\lb^k), \lb^k)$ on both sides of the preceding relation, we have the following upper bound on term (B):
\#\label{eq:termb}
&L(\xb ^{k+1}, \lb^{k}) - L(\xb ^{*}(\lb^k), \lb^k) \\
&\quad \le \Delta_{\xb }^k -
\|\nabla_\xb  L(\xb ^{k}, \lb^{k})\|^2_{AP^k} + \frac{\ell_L}{2}\|AP^k\nabla_\xb L(\xb ^k,\lb^k)\|^2. \notag
\# 
\textit{Term (C).} Since for any $\lb$, $L(\xb ^*(\lb), \lb) = g(\lb)$, we have
\#\label{eq:termc}
L(\xb ^*(\lb^{k}), \lb^{k}) - L(\xb ^*(\lb^{k+1}), \lb^{k+1}) &= g(\lb^k) - g(\lb^{k+1}) \notag\\
&= \Delta_{\lb}^k - \Delta_{\lb}^{k+1}. 
\#
Finally, by substituting \eqref{eq:term_a_1_result}, \eqref{eq:termb}, and \eqref{eq:termc} into \eqref{eq:deltaL}, we have
\$
\Delta_{\xb }^{k+1} &\le \Delta_{\xb }^k  +3\|\nabla g(\lb^k)\|^2_{BD^k}  +\Delta_{\lb}^k - \Delta_{\lb}^{k+1}\notag\\
&\quad - \nabla_\xb  L(\xb ^{k}, \lb^{k})^\intercal\Big[AP^k - \big(2\beta+\frac{\ell_L}{2}\big)A^2(P^k)^{2} \notag\\
&\qquad-  \frac{12\beta}{s^2}I\Big]\nabla_\xb  L(\xb ^{k}, \lb^{k}).
\$
This concludes the proof of the proposition.
\end{proof}

\section{Proof of Theorem \ref{thm:strong}}
In this section, we prove Theorem \ref{thm:strong}.

\begin{proof}[Proof of Theorem \ref{thm:strong}]
We first note that if dual stepsizes satisfy \eqref{eq:steps}, for $\beta$ defined in \eqref{eq:beta}, we have
\#\label{eq:dual_step_cond}
\beta &= \max_{i\in\cN}\{b_i\oq_i\} \le \min\big\{\frac{1}{16\ell_g}, \frac{\underline{\alpha}s^2}{60} \big\}.
\#
Then if primal stepsizes satisfy \eqref{eq:steps}, using $\beta\le 1/(16\ell_g)$ and $\ell_L = \ell + 2\mu$ defined in Lemma \ref{lem:L_property}, we have
\#\label{eq:primal_step_cond}
\big\|AP^k\big\| &= \max_{i\in\cN}\{ a_i\|P_i^k\|\} \notag\\
&\le \max_{i\in\cN}\{a_i\op_i\} \notag\\
&\le \frac{1}{2[1/(4\ell_g) + \ell_L]} \notag\\
&\le \frac{1}{2(4\beta + \ell_L)}.
\#
Next, we show a bounds on a matrix $M^k = D^k - 16(1 + 2\beta\ell_g)\beta/s^2 I$ related to $D^k$ defined in Proposition \ref{prop:Lg}.
By the definition of $D^k$ in Proposition \ref{prop:Lg}, the matrix $M^k$ satisfies
\# \label{eq:qk1}
M^k & = AP^k - (2\beta+\frac{\ell_L}{2})A^2(P^k)^{2}- \frac{(28 + 32\beta\ell_g)\beta}{s^2}I \notag \\
& \succeq \frac{1}{2}AP^k - (2\beta+\frac{\ell_L}{2})A^2(P^k)^{2} \notag\\
& = \frac{1}{2} A^{\frac{1}{2}}(P^k)^{\frac{1}{2}}\sbr{I- (4\beta+\ell_L)AP^k}A^{\frac{1}{2}}(P^k)^{\frac{1}{2}}, 
\#
where the inequality is due to $16\ell_g\beta<1$ and $\beta\le\underline{\alpha}s^2/60$ in \eqref{eq:dual_step_cond}. Thus, by \eqref{eq:qk1}, the smallest eigenvalue of $M^k$ satisfies
\#\label{eq:bd_qk}
&\theta_{\min}(M^k) \notag\\
&\ \ge \frac{1}{2} \theta_{\min}\big(A^{\frac{1}{2}}(P^k)^{\frac{1}{2}}\sbr{I- (4\beta+\ell_L)AP^k}A^{\frac{1}{2}}(P^k)^{\frac{1}{2}}\big) \notag\\
&\ \ge \frac{1}{2} \sbr{1- \rbr{4\beta+\ell_L}\|AP^k\|} \theta_{\min}\rbr{AP^k} \notag \\
&\ \ge \frac{1}{4} \theta_{\min}\rbr{AP^k} \ge \frac{\underline{\alpha}}{4},
\#
where the third and the last inequalities are due to \eqref{eq:primal_step_cond} and \eqref{eq:under_a_b}, respectively.

Now we combine Propositions \ref{prop:gL} and \ref{prop:Lg} to show the result. By multiplying Proposition \ref{prop:gL} by $8$ and adding Proposition \ref{prop:Lg}, we have
\#\label{eq:step_1_and_2}
9\Delta_{\lb}^{k+1} + \Delta_{\xb}^{k+1} &\le  9\Delta_{\lb}^{k}  - (1-8\beta\ell_g)\|\nabla g(\lb^{k})\|^2_{BQ^k} \notag\\ 
&\quad + \Delta_{\xb}^k - \|\nabla_\xb L(\xb^{k}, \lb^{k})\|^2_{M^k}, 
\#
where $M^k = D^k - 16(1 + 2\beta\ell_g)\beta/s^2 I$. By the PL inequality that $g(\lb)$ satisfies in Lemma \ref{lem:g} with $p_g = (1-\gamma)/(\ell+2\mu)$, we have
\#\label{eq:pl_g}
\|\nabla g(\lb)\|^2 \ge 2p_g\big(g(\lb^*) - g(\lb)\big) = \frac{2(1-\gamma)}{\ell + 2\mu}\Delta_{\lb}^k
\#
Thus, by \eqref{eq:dual_step_cond} that $16\beta\ell_g<1$, we have
\# \label{eq:gap_k} 
(1-8\beta\ell_g)\|\nabla g(\lb^{k})\|^2_{BQ^k} 
&\ge \frac{1}{2}\|\nabla g(\lb^{k})\|^2_{BQ^k} \notag \\
&\ge \frac{1}{2}\theta_{\min}(BQ^k)\|\nabla g(\lambda^{k})\|^2 \notag \\
& \ge \frac{(1-\gamma)\underline{\beta}}{\ell+2\mu}\Delta_{\lb}^k,
\#
where the last inequality follows from \eqref{eq:under_a_b} and \eqref{eq:pl_g}. Similarly, by the $s$-strong convexity of $L(\cdot, \lb^k)$ in Lemma \ref{lem:L_property}, we have
\#\label{eq:s2}
\|\nabla_\xb L(\xb^k, \lb^k)\|^2 &\ge 2s\rbr{L(\xb^k, \lb^k) - L(\xb^*(\lb^k), \lb^k)} \notag\\ 
&= 2s\Delta_\xb^k.
\#
Thus, by the lower bound of $M^k$ in \eqref{eq:bd_qk}, we have
\#\label{eq:l_k}
\|\nabla_\xb L(\xb^{k}, \lb^{k})\|^2_{M^k} &  \ge \theta_{\min}(M^k)\|\nabla_\xb L(\xb^k, \lb^{k})\|^2 \notag \\
&  \ge \frac{\underline{\alpha}}{4}\|\nabla_\xb L(\xb^k, \lb^k)\|^2 \notag \\
&  \ge  \frac{\underline{\alpha}}{2}s\Delta_\xb^k,
\#
where the last inequality follows from \eqref{eq:s2}.

Finally, substituting \eqref{eq:gap_k} and \eqref{eq:l_k} into \eqref{eq:step_1_and_2},  we have
\$
& 9\Delta_{\lb}^{k+1} + \Delta_{\xb}^{k+1} \\
&\quad \le  9\Delta_{\lambda}^{k}  - \frac{(1-\gamma)\underline{\beta}}{\ell+2\mu}\Delta_{\lb}^k + \rbr{1 -\frac{s\underline{\alpha}}{2}}\Delta_{\xb}^k \\
&\quad\le 9\sbr{1-\frac{(1-\gamma)\underline{\beta}}{9(\ell+2\mu)}}\Delta_{\lb}^{k} + \rbr{1 - \frac{s\underline{\alpha}}{2}}\Delta_{\xb}^k \notag \\
&\quad\le \rbr{1-\rho}\rbr{9\Delta_{\lb}^{k} + \Delta_{\xb}^k},
\$
where $\rho = \min \{(1-\gamma)\underline{\beta}/[9(\ell+2\mu)], s\underline{\alpha}/2 \}$. Thus, let $\Delta^k = 9\Delta_{\lb}^{k} + \Delta_{\xb}^k$ as defined in \eqref{eq:merit}, we have
\$
\Delta^{k+1} \le (1-\rho)\Delta^k.
\$
This concludes the proof of our theorem.
\end{proof}

\section{Proof of Corollary \ref{coro:dish}}
%Before showing the analysis, we note that by the definition that $\xb^*(\lb) = \argmin_{\xb} \{f(\xb) + \lb^{\intercal}W \xb + \mu\xb^{\intercal}V\xb/2\}$, vectors $\lb$ with the same value of $W\lb$ correspond to the same minimizer $\xb^*(\lb)$. Thus, for convenience, we denote $\xb^*[W\lb] = \xb^*(\lb)$ for any $\lb$.
The next lemma states the Lipschitz continuity of $\xb^*(\lb)$.
\begin{lemma}\label{lem:lip_x_star} 
Under Assumption \ref{ass:hessian}, for any $\lb,\bm{\eta}\in\RR^{nd}$, we have
\$
\|\xb^*(\lb) - \xb^*(\bm{\eta})\| \le \frac{1}{s}\|W\lb- W\bm{\eta}\|.
\$
\end{lemma}
\begin{proof}
The result is well-known from properties of implicit functions. For the sake of completeness, we provide the proof here. We first show that for any $\lb\in\RR^{nd}$, we have
\#\label{eq:nabla_M}
\nabla\xb^*(\lb) = -[\nabla_{\xb\xb}^2L(\xb^*(\lb),\lb)]^{-1}W.
\#
Due to the optimality condition of the inner problem that $\nabla_{\xb}L(\xb^*(\lb),\lb) = 0$ for any $\lb\in\RR^{nd}$, by taking derivatives with respect to $\lb$ on both sides, using the chain rule and the implicit function theorem, we have
\$
\nabla_{\xb\xb}^2L(\xb^*(\lb),\lb)\nabla\xb^*(\lb) + \nabla_{\xb\lb}^2L(\xb^*(\lb),\lb)= 0.
\$
By rearranging the terms in the above relation and substituting $\nabla_{\xb\lb}^2L(\xb^*(\lb),\lb)=W$, we have \eqref{eq:nabla_M}. Thus, for any $\lb,\bm{\eta}\in\RR^{nd}$, by the mean value theorem, there exists $\alpha \in [0,1]$ such that $\overline{\lb} = \alpha \lb + (1-\alpha) \bm{\eta}$ satisfies,
\$
\xb^*(\lb) - \xb^*(\bm{\eta}) &= \nabla \xb^*(\overline{\lb})(\lb - \bm{\eta}) \\
& = -\big[\nabla_{\xb\xb}^2L(\xb^*(\overline{\lb}),\overline{\lb})\big]^{-1}(W\lb - W\bm{\eta}),
\$
where the last equation follows from \eqref{eq:nabla_M}. Thus, we have
\$
\|\xb^*(\lb) - \xb^*(\bm{\eta})\| &= \big\|\big[\nabla_{\xb\xb}^2L(\xb^*(\overline{\lb}),\overline{\lb})\big]^{-1}(W\lb - W\bm{\eta})\big\| \\
&\le \big\|\big[\nabla_{\xb\xb}^2L(\xb^*(\overline{\lb}),\overline{\lb})\big]^{-1}\big\|\|W\lb - W\bm{\eta}\| \\
&\le \frac{1}{s}\|W\lb - W\bm{\eta}\|,
\$
where the last inequality follows from the $s$-strong convexity of $L(\cdot,\lb)$ in Lemma \ref{lem:L_property}. This concludes the proof.
%Since Lemma \ref{lem:L_property} shows that $L(\cdot, \lb)$ is $s$-strongly convex for any $\lb\in\RR^{nd}$ and $\|\nabla_{\xb\lb}^2L(\xb,\lb)\| = \|W\|\le2$, the result follows from Lemma 2.2 b) in \cite{ghadimi2018approximation}.
\end{proof}
Now we prove Corollary \ref{coro:dish}.
\begin{proof}
Following from Theorem \ref{thm:strong}, we have 
\#\label{eq:Delta_r_linear}
\Delta^k \le (1-\rho)^k \Delta^0,
\#
where $\Delta^k = 9\Delta_{\lb}^{k} + \Delta_{\xb}^{k}$ is defined in \eqref{eq:merit}. Now we take a close inspection on $\Delta_{\lb}^{k}$ and $\Delta_{\xb}^{k}$, respectively. By the definition in \eqref{eq:terrors}, we have
\#\label{eq:Delta_lb}
\Delta_{\lb}^{k} &= g(\lb^*) - g(\lb^k) \notag \\
& = \tilde f^*(-W\lb^k) - \tilde f^*(-W\lb^*) \notag \\
& \ge \frac{s_{\tilde f^*}}{2} \|W\lb^k - W\lb^*\|^2 \notag\\
& \ge \frac{s_{\tilde f^*}s^2}{2}\|\xb^*(\lb^k) - \xb^*(\lb^*)\|^2,
\#
where the second equality is due to $g(\lb) = -\tilde f^*(-W\lb)$ for any $\lb$, the first inequality follows from the $s_{\tilde f^*}$-strong convexity of $\tilde f^*$ in Lemma \ref{lem:penalized_f_conjugate}, and the last inequality follows from Lemma \ref{lem:lip_x_star}. By the definition in \eqref{eq:terrors}, we also have
\#\label{eq:Delta_xb}
\Delta_{\lb}^{k} &= L(\xb^k,\lb^k) - L(\xb^*(\lb^k), \lb^k) \notag \\
&\ge \frac{s}{2}\|\xb^k - \xb^*(\lb^k)\|^2,
\#
where the inequality is due to the $s$-strong convexity of $L(\cdot,\lb^k)$ in Lemma \ref{lem:L_property} for any $\lb^k$. By substituting \eqref{eq:Delta_lb} and \eqref{eq:Delta_xb} in the definition of $\Delta^k$ in \eqref{eq:merit}, we have
\#\label{eq:metrit_bd_x}
\Delta^k &= 9\Delta_{\lb}^{k} + \Delta_{\xb}^{k} \notag \\
&\ge 9\frac{s_{\tilde f^*}s^2}{2}\|\xb^*(\lb^k) - \xb^*(\lb^*)\|^2 + \frac{s}{2}\|\xb^k - \xb^*(\lb^k)\|^2 \notag \\
&\ge \frac{s}{2}\min\big\{1,9s_{\tilde f^*}s\big\}\big[\|\xb^*(\lb^k) - \xb^*[\bm{\zeta}^{\textsf{OPT}}]\|^2 \notag \\
&\qquad+ \|\xb^k - \xb^*(\lb^k)\|^2 \big]\notag \\
&\ge \frac{s}{4}\min\big\{1,9s_{\tilde f^*}s\big\}\|\xb^k - \xb^*[\bm{\zeta}^{\textsf{OPT}}]\|^2, 
\#
where the last inequality follows from the inequality that $2(\|a\|^2 + \|b\|^2) \ge \|a+b\|^2$ for any $a,b\in\RR^{nd}$.
By substituting \eqref{eq:metrit_bd_x} in \eqref{eq:Delta_r_linear}, we have
\$
\frac{s}{4}\min\big\{1,9s_{\tilde f^*}s\big\}\|\xb^k - \xb^*[\bm{\zeta}^{\textsf{OPT}}]\|^2 \le (1-\rho)^k\Delta^0.
\$
By rearranging terms in the above relation and using $s_{\tilde f^*} = 1/\ell_L$ in Lemma \ref{lem:L_property}, we have
\$
\|\xb^k - \xb^*[\bm{\zeta}^{\textsf{OPT}}]\|^2 \le c(1-\rho)^k,
\$
where $c = 4\ell_L\Delta^0/[s\cdot\min\{\ell_L,9s\}]$. We note that $\xb^*(\lb^*) = \xb^{\textsf{OPT}}$ due to the strong duality. Therefore, this concludes the proof of the corollary.
\end{proof}

\end{document}